\documentclass[12pt,reqno]{amsart}
\usepackage[margin=1.37in]{geometry}

\usepackage[english]{babel}
\usepackage{amsmath}
\usepackage{amssymb}
\usepackage{amsthm}
\usepackage{graphicx,color} 
\usepackage[dvipsnames]{xcolor}
\usepackage{bm,bbm}
\usepackage{graphicx} %Loading the package
\graphicspath{{pics/}} %Setting the graphicspath

\usepackage[normalem]{ulem}
\usepackage{subfiles} % Best loaded last in the preamble
\usepackage{enumitem}

\usepackage{soul}
\setstcolor{red}

\usepackage{hyperref}

\definecolor{lgray}{gray}{0.9}
\definecolor{llgray}{gray}{0.95}
\definecolor{lllgray}{gray}{0.975}

\newcommand{\black}{\color{black}}
\newcommand{\blue}{}
\date{}

\newcommand{\bk}{{\bm k}}

\newcommand{\br}{{\bm r}}

 \newcommand{\cR}{{\mathcal R}}

 \renewcommand{\restriction}{|}

\newcommand{\Enorm}[1]{\| #1\|_r}

\newcommand{\EnormZ}[1]{\| #1\|_{0}}

\newcommand{\sX}{{\mathsf X}}

\newcommand{\XM}{{\mathsf X}_M}
\newcommand{\XN}{{\mathsf X}_N}

\newcommand{\per}{{\rm per}}

\newcommand{\Lt}{{\sf L^2_\per}}

\newcommand{\Hss}[1]{{\sf H}^{#1}_\per}
\newcommand{\Hs}[1]{{\sf H}^{#1}_\per}

\newcommand{\C}{\mathbb C}
\newcommand{\N}{\mathbb{N}}

\newcommand{\Z}{\mathbb Z}

\newcommand{\calH}{{\mathcal{H}}}

\newcommand{\R}{\mathbb R}
\newcommand{\cH}{\mathcal{H}}

\newcommand{\nc}{\newcommand}
\nc{\e}{\epsilon}
\nc{\be}{\beta}
\nc{\del}{\delta}
\nc{\G}{\Gamma}
\nc{\g}{\gamma}
\nc{\gam}{\gamma}
\nc{\ka}{\kappa}
\nc{\lam}{\lambda}
\nc{\Lam}{\Lambda}
\nc{\Om}{\Omega}
\nc{\om}{\omega}
\nc{\ta}{\tau}
\nc{\w}{\omega}
\nc{\io}{\iota}
\nc{\h}{\theta}
\nc{\z}{\zeta}
\nc{\si}{\sigma}
\nc\vphi{{\varphi}}
\nc\eps{\epsilon}

\newcommand{\lan}{\langle}
\newcommand{\ran}{\rangle}

\newcommand{\id}{{\blue{\rm Id}}}
\newcommand{\one}{\id}

\newcommand{\Ran}{{\rm{Ran\, }}}

\newcommand{\re}{\operatorname{Re}}

\newcommand{\ls}{\lesssim}

\newcommand{\ra}{\rightarrow}

\newcommand{\VERSIONS}[1]{}

\renewcommand{\Pr}{{\sf P}}
\newcommand{\PN}{\Pr_{\!N}}
\newcommand{\PNp}{\Pr_{\!N}^\perp}
\newcommand{\PM}{\Pr_{\!M}}
\newcommand{\PMp}{\Pr_{\!M}^\perp}
\newcommand{\PNM}{\Pr_{\!M}^{\!N}}

\newcommand{\HMp}{\mathcal{H}_{\! M}^\perp}

\newcommand{\Hcal}{\mathcal H}

\newcommand{\HM}{\mathcal{H}_{\! M}}
\newcommand{\HMN}{\mathcal{H}_{\! M}^N}
\newcommand{\HMl}{\mathcal{H}_{\! M}(\lax)}

\newcommand{\Hsi}{\mathcal{H}_{\! \sigma}}

\newcommand{\Gza}{(H_{0}+\alpha)}
\newcommand{\Hza}{H_{0,\alpha}}
\newcommand{\Hzl}{H_{0,-\lax}}

\newcommand{\VMN}{V_{\! M}^N}

\newcommand{\UM}{U_{\! M}}
\newcommand{\UMl}{U_{\! M}(\lax)}

\newcommand{\Us}{U_{\! \sigma}}
\newcommand{\Usl}{U_{\! \sigma}(\lax)}
\newcommand{\UMN}{U_{\! MN}}
\newcommand{\UMNl}{U_{\! MN}(\lax)}

\newcommand{\QM}{{\sf Q}_{\!M}}

\newcommand{\GNM}{{\sf G}_{\!M}^N}

\newcommand{\Pz}{P_0}
\newcommand{\oP}{\Pz^\perp}

\newcommand{\nuN}{\nu_N}
\newcommand{\nuMi}{\nu_{M i}}

\newcommand{\nusi}{\nu_{\sigma i}}

\newcommand{\las}{\lambda_\star}
\newcommand{\lac}{\lambda_{\circ}}
\newcommand{\lat}{\lambda_{\sf t}}

\newcommand{\laxs}{\lambda_{\sigma}}
\newcommand{\laxsi}{\lambda_{\sigma i}}

\newcommand{\la}{{\lambda_\star}}

\newcommand{\laN}{\rho_N}
\newcommand{\laM}{\rho_M}

\newcommand{\lax}{{\lambda}}

\newcommand{\I}{I_0}
\newcommand{\aaa}{\alpha}
\newcommand{\bb}{\alpha}
\newtheorem{remark}{Remark}
\newtheorem{prop}{Proposition}
\newtheorem{lem}{Lemma}
\newtheorem{thm}{Theorem}
\newtheorem{cor}[thm]{Corollary}
\newtheorem{assumption}{Assumption}

\newcommand{\new}[1]{\blue {#1} \black }
\newcommand{\old}[1]{}

\title[The FSM and computation of eigenvalues. Main result]{\new{Analysis of the Feshbach--Schur method for the Fourier Spectral discretizations of Schr\"odinger operators}}

\author{Genevi\`eve Dusson}
\address{Laboratoire de Math\'ematiques de Besan\c{c}on, UMR CNRS 6623, Universit\'e Bourgogne Franche-Comt\'e,  16 route de Gray, 25030 Besan\c{c}on, France}
\email{genevieve.dusson@math.cnrs.fr}

\author{Israel Michael Sigal}
\address{Department of Mathematics,
University of Toronto,
40 St. George Street,  Bahen Centre, 
Toronto, ON M5S 2E4, 
Canada}
\email{im.sigal@utoronto.ca}

\author{Benjamin Stamm}
\address{Applied and Computational Mathematics, Department of Mathematics, RWTH Aachen University, Schinkelstr. 2, 52062 Aachen, Germany}
\email{best@acom.rwth-aachen.de}

\begin{document}

\begin{abstract}
    In this article, we propose a new numerical method and its analysis to solve eigenvalue problems for self-adjoint Schr\"odinger operators, by combining the Feshbach--Schur perturbation theory with \old{planewave}\new{the spectral Fourier} discretization. 
    In order to analyze the method, we establish an abstract framework of  Feshbach--Schur perturbation theory with minimal regularity assumptions on the potential that is then applied to the setting of the new \old{planewave}\new{spectral Fourier} discretization method.
    Finally, we present some numerical results that underline the theoretical findings.
\end{abstract}

\maketitle

% \tableofcontents
\section{Introduction}

\new{
In this article, we address the problem of the computation of eigenvalues of  self-adjoint Schr\"odinger operators 
(quantum Hamiltonians) of the form 
\[
    \Hcal=-\Delta +V.
\]
Our main results are a priori error estimates for the approximation error of the eigenvalue and eigenfunction for a new method allowing irregular potentials beyond the regularity assumptions within the standard variational setting~\cite{babuvska1989finite,Babuska1991-cg,chatelin2011spectral}. 
The main ingredient of this method and its analysis is a reduction of this infinite-dimensional problem to a finite-dimensional one in a fully controlled way with an effective estimate of the error terms, using the Feshbach--Schur map (FSM) method. This method originated in works of I. Schur on the Dirichlet problem in planar domains and H. Feshbach, on  resonances in nuclear physics, and was then developed independently in numerical analysis, computational quantum chemistry and mathematical physics, see \cite{Griesemer2008-uw,Gustafson2011-wd} with the original techniques called variously the Feshbach projection and Schur complements methods. 
}

We combine this approach with 
\new{spectral Fourier} discretizations which are widely used in numerical methods in electronic structure calculation,
especially for condensed matter simulations and in materials science, \new{and known in this community as planewave discretization.}
Electronic structure calculation is indeed one of the problems we have in our sight.
And one particular very useful aspect of planewaves is that they are eigenfunctions of the Laplace operator,  \new{entering} the Hamiltonian $\Hcal=-\Delta +V$ that needs to be diagonalized in order to determine the electronic structure of the system.

\new{
Our analysis is thus relying on non-variational perturbative techniques. 
Perturbation-based approaches have a long history in quantum mechanics.
}
For example, different perturbation methods have been proposed, such as~\cite{Brust1964-hw, Moller1934-iz}, traditionally to introduce more physical details e.g. many-particle interactions in a given approximation. 
\new{
The mathematical justification has been provided by the seminal work of Kato~\cite{Kato1976-hm}.
Perturbation methods have also been used to study van der Waals interactions between two hydrogen atoms in the dissociation limit in a mathematically rigorous way in
\cite{cances2020van,cances2018van}.} 
More recently, a post-processing strategy has been proposed by some of the authors for planewave discretizations for non-linear eigenvalue problems~\cite{Cances2014-lb,Cances2016-vy,Cances_undated-fd,Dusson2017-jg}, which considers the exact solution as a perturbation of the discrete (using the planewave basis) approximation.

This is in spirit not so far from so-called two-grid methods, where a first problem is solved on a coarse basis, i.e. in a small discretization space, and a small problem is solved on a fine basis. \new{In the case of eigenvalue problems using a same Hamiltonian $\Hcal=-\Delta +V$ as in this article, two-grid and three-grid methods have been proposed e.g. in \cite{dai2008three,Xu1999-vo} within a variational approximation. }
A two-grid method has also been proposed for nonlinear eigenvalue problems of a Gross--Pitaevskii type equation in ~\cite{Cances2018-ow}.

\new{As emphasized above, we extend in this article} the FSM-method to establish finite-dimensional approximations \new{based on the spectral Fourier basis} to solve the Schr\"odinger eigenvalue problem with controlled errors on the eigenvalues and eigenvectors. 
To be a little more concrete, we define a new problem in a coarse, \new{finite-dimensional, subspace} $\XM\subset \sX$ \new{spanned by Fourier modes} yielding the exact eigenvalue one would obtain when computing it in the infinite-dimensional space~$\sX$.
Indeed, our contribution follows a new Ansatz based on the question: Can we find a discrete Hamiltonian acting on the finite-dimensional space $\XM$ that has the exact eigenvalue $\las$ of the original Hamiltonian (acting on $\sX$) as eigenvalue? It turns that the answer is yes, but that the discrete Hamiltonian depends itself on $\las$, through the \new{Feshbach--Schur (FS)}-map, leading to an eigenvalue problem in~$\XM$ that is nonlinear in the spectral parameter. 
Not surprisingly, the map cannot be computed exactly but only be approximated through a fast decaying series, that is truncated based on a parameter $K$, and which requires computations in a larger space $\XN$ with $\XM\subset \XN \subset \sX$. 
\new{However, this defines a new numerical method that is defined by the three parameters $\sigma=(N,M,K)$ and that can be rigorously analysed.}
In this work we quantify the error introduced due to the discretization parameters \old{$K,N,M$}\new{$\sigma=(N,M,K)$.}
\old{and show that the eigenvalue and eigenfunction errors are bounded by two terms: i) a term with algebraic decay with respect to the truncation parameter $K$, and ii) a term with a regularity-dependent convergence rate in $N$.
We also quantify the explicit dependency of the error in terms of the parameter $M$ defining the discrete space $\XM$ and \new{the regularity of } the potential.\old{, including its regularity.}
Our analysis  reveals that the algebraic decay rate with respect to $K$ increases with increasing~$M$.}

\old{
Our method uses an adapted version of perturbation theory based on a slightly more regular notion of relative form-boundedness, as stated in Assumption~\ref{as:pot}, developed as an abstract theory in Section~\ref{sec:pert-est}, which thus only requires little regularity of the potential including cases which are not covered by the standard analysis of \old{planewave}\new{spectral Fourier} discretizations.
We also illustrate our approach by computing eigenvalues of several 1D Schr\"odinger operators.
}

This article is organized as follows. In Section~\ref{sec:sec2} we present the problem and numerical method that is used to find approximations thereof, as well as the main approximation result of the article and the error bounds on the eigenvalues.
Section~\ref{sec:pert-est} provides the above-mentioned abstract framework of Feshbach--Schur perturbation theory based on the regularized version of form-boundedness whereas Section~\ref{sec:prelim-res} contains some technical results needed to prove the main result which follows in Section~\ref{sec:main-proof}. 
Finally, we present in Section~\ref{sec:NumRes} some numerical results to illustrate the convergence as well as the error bounds, and we conclude with some perspectives in Section~\ref{sec:sec6}.

\section{Set-up and results}
\label{sec:sec2}
\subsection{Problem statement}\label{sec:probl}
 In order to simplify the notation, we consider a cubic lattice $\cR=L\Z^d$ ($L > 0$, $d=1,2,3$), 
 but all our arguments straightforwardly apply to the general case of any Bravais lattice. 
In this paper we are interested in the spectral theory of the self-adjoint Schr\"odinger operators
(quantum Hamiltonians) 
\[
	\Hcal := -\Delta + V,
\]
with reasonably regular, $\cR$-periodic potentials $V$, acting on 
 the Hilbert  space 
\begin{align*}\label{L2per}	 \Lt
	&:=\left\{ u \in L^2_{\rm loc}(\R^d) \; \middle| \; u \mbox{ is $\cR$-periodic} \right\},
\end{align*}
endowed with the scalar product
$ 
\langle u,v\rangle = \int_\Omega u(\br) \,
v(\br) \, d\br$ and the induced norm $\|\cdot\|$, where  $\Omega=[0,L)^d$ is the chosen fundamental cell of the lattice $\cR=L\Z^d$.
\new{Note that $V$ acts as a multiplicative operator, whereas $V$ is either a function (in the more regular case) or a distribution (in the less regular case), whose regularity will be discussed shortly.}

Specifically, we would like to solve the eigenvalue problem
\begin{equation}
	\label{eq:EVP}
	\Hcal\varphi = \lam\varphi,
\end{equation}
in a space $\sX\subset \Hs{1}$. Here, $\Hs{1}$ is the  Sobolev space of index 1 of periodic functions on $\Omega$, which is defined in precise terms later on and equation~\eqref{eq:EVP} is considered in the weak sense.

 To this end we use the Feshbach--Schur method to reduce the problem to a finite dimensional one.
To simplify the exposition, we will assume that the eigenvalue of interest $\la$ is isolated, which is true for the smallest eigenvalue under fairly general assumptions of $V$, \new{see~Theorems XIII.46 - XIII.48 of the textbook \cite{RSIV}.}
We denote by $\|\cdot\|$ the operator norm on ${\mathcal L}(\Lt)$, the space of bounded linear operators on $\Lt$.
To formulate our condition on the potential $V$, we introduce the following norm measuring its regularity
\[
    \Enorm{V} := \|(-\Delta+\id)^{-1/2+r/2} V (-\Delta+\id)^{-1/2+r/2}\|,
\]
{where the operator $(-\Delta+\id)^{s}$ is defined by the Fourier transform (cf. Appendix \ref{sec:tech-res}). In what follows, we thus assume that the potential $V$ satisfies the following condition.}

\begin{assumption}
   \label{as:pot}
    The potential $V$ is real, $\cR$-periodic and satisfies  
    \[
        \Enorm{V}  < \infty \quad \text{for some} \quad r > 0. 
    \]
\end{assumption}
 Assumption \ref{as:pot}  implies that $V$ is $\Delta$-form bounded~\cite{CFKS,RSII}, which corresponds to $r=0$. The latter, weaker property implies that $\Hcal$ (a) is self-adjoint; (b) is bounded below and (c) has purely discrete spectrum (see e.g. \cite{CFKS,RSII,RSIV, HS}). 
Moreover, potentials $V$ belonging to the Sobolev spaces,  $\Hs{s}:=(-\Delta+1)^{s/2}\Lt$ satisfy this assumption as shown in Appendix~\ref{sec:tech-res}, Lemma~\ref{lem:Hs-bnd} for $r \le s+1$ and $r<1+\frac{s}{2}-\frac{d}{4}$.
In terms of  Sobolev spaces, Assumption~\ref{as:pot} states that $V$, as an operator, maps $\Hs{1-r}$ into $\Hs{-1+r}$.

\new{
As an example, we note that the Coulomb potential $V(x)=\frac{Z}{|x|}$  in three dimensions ($d=3$) satisfies $V\in \Hs{s}$, $s<1/2$. 
Thus, in view of the above discussion we conclude that 
 Assumption \ref{as:pot} is satisfied for any $r< 1/2$.
 For potentials $V\in\Lt$, thus with $s=0$, the above discussion yields that Assumption \ref{as:pot} is satisfied for any $r<1-\frac{d}{4}$.
}
 
\subsection{Approach} 
In our approach, we reduce the exact infinite dimensional eigenvalue problem to a finite dimensional one in a controlled way for fairly irregular potentials. Of course, we have to pay a price for this, which is that at one point we solve a one-dimensional fixed point problem
that can be equivalently seen as a non-linear eigenvalue problem.
A key ingredient of our method is the finite dimensional space and the corresponding orthogonal projection onto which we map the original problem to obtain a reduced, finite-dimensional one.

Let  $\XM$ \new{denote the subspace of $\Lt$ spanned by the eigenfunctions of $-\Delta$ on $\cR$, with eigenvalues smaller than $\laM$, as}
\new{\begin{align}
	\label{XM}
    \XM = \left\{ \sum_{k\in \mathcal R^*, |k| \le M} \hat u_k e_k(x) \;\middle|\; \hat u_k^* = \hat u_{-k}, \; \hat u_k 
    \in \mathbb C \right\},
    %\quad \; 
    \end{align}
}
\new{where $e_k(x) = |\Omega|^{-\frac12} e^{ik\cdot x}$} (\new{Fourier modes, also called } planewaves),   $ \mathcal R^* = \frac{2\pi}{L}\mathbb Z^d,$ and %thus where
\[
   \laM:=\left(\frac{2\pi M}{L}\right)^2.
\]

Let $\PM$ be the $\Lt$-orthogonal projection onto $\XM$ and $\PMp:=\id-\PM$. 
We consider the Galerkin approximation 
of the linear Hamiltonian
$	\Hcal := -\Delta + V,$
\[
	\HM := \PM(-\Delta +  V) \PM.
\]
\new{
Let $\varphi$ denote an eigenfunction of \eqref{eq:EVP},} 
\old{We now}
introduce the projections $\varphi_M = \PM \varphi$ and $\varphi_M^\perp =\PMp\varphi$ and project 
the exact eigenvalue problem \eqref{eq:EVP}
 onto the subspace $\XM$ and its complement $\XM^\perp$ to obtain 
\begin{align}
	\label{eq:Seq1}
	\PM( \HM - \lam) \varphi_M &= -\PM V \varphi_M^\perp, \\
	\label{eq:Seq2}
	\PM^\perp(\HMp- \lam)\varphi_M^\perp &= -\PMp V\varphi_M,
\end{align}
where  $\HMp := \PMp \Hcal\PMp$. 
\new{Here and in the remainder of this article, we abuse notation and write $\lambda$ instead of $\lambda\id$ in order to denote the multiplicative operator.}
Next, in Appendix~\ref{sec:tech-res}, we prove the following 
\begin{lem}
\label{lem:Hperp-low-bnd}
Let Assumption~\ref{as:pot} hold  and define $\ka_M:=\laM - (\laM+1)  \,  \laM^{-{r}}\Enorm{V}$.
Then
\begin{align}
    \label{HNlowbnd}
  \HMp 
    &\ge \ka_M \ \text{ on }\ \Ran \PMp.
\end{align}
\end{lem} 
\new{Here, $\Ran$ denotes the range (or image) of the following operator.}

Thus for $\lax < \ka_M$,  
the operator $\HMp - \lax$ is invertible and
we can solve~\eqref{eq:Seq2} for $\varphi_M^\perp$ and thus $\varphi_M^\perp=- (\HMp - \lam)^{-1}\PMp V\varphi_M$. Substituting the result into~\eqref{eq:Seq1}, we obtain the
non-linear eigenvalue problem 
\begin{equation}
	\label{EVPN}
	\big( \HM + \UM(\lam)\big) \varphi_M = \lam \varphi_M,
\end{equation}
where we introduced the \emph{effective interaction} $\UM(\lax):\XM\to \XM$, or a Schur complement,
\begin{align}
	\label{UNlam-def}
	\UM(\lam)
	:=
	-\PM V\PMp (\HMp- \lax)^{-1} \PMp V\PM.
\end{align}

We then have the following proposition, which is proved in Appendix~\ref{sec:tech-res}.
\begin{prop}
\label{prop:UN-well-defined} 
 For each $\lax$ such that $\lax <  \ka_M$,  $\UM(\lax)$ is a  well-defined 
 operator as a product of three maps: $\PM, V$ and $\PMp ( \HMp - \lax)^{-1} \PMp$ between various but matching Sobolev spaces. 
  \end{prop}

Now, we construct a completely computable approximation of the eigenvalue problem~\eqref{EVPN}, with the operators involved being  sums of products of finite matrices.
Namely, we expand the resolvent $(\HMp- \lax)^{-1}_{|\text{Ran}\PMp} = ( -\Delta + V_M^\perp- \lax)^{-1}_{|\text{Ran}\PMp}$ in \eqref{UNlam-def} in the formal Neumann series in $V_M^\perp\new{:= \PMp V\PMp}$, then  truncate this series at $K\in\N$ and replace the projections $\PMp=\one-\PM$ by $\PNM:=\PN-\PM$, with $N>M$. Introducing the notation 
\begin{equation}
    \label{eq:GNM22}
    \GNM(\lax) := ( -\Delta - \lax)|_{\Ran \PNM}^{-1}
\end{equation}
and $\VMN:= \PNM V \PNM$, we obtain the following
 truncated effective interaction
\begin{align}\label{UNMKlam}
	\Us(\lax)
	:=
	-\PM V \PNM
	 R_{\sigma}(\lax)  \PNM V \PM,
\end{align}
where $\sigma:=(N, M, K)$ and $R_{\sigma}(\lax) :=\sum_{k=0}^K (-1)^{k}\Big[ \GNM(\lax)\VMN \Big]^k \GNM(\lax) $. 
 Since all the operators involved in \eqref{UNMKlam} are finite matrices, this family is well-defined and computable.
Now, we define $\Hsi(\lax):=\HM + \Us(\lax)$ on $\XM$ and consider the eigenvalue problem: 
find an eigenvalue $\laxsi$ and the corresponding eigenfunctions $\varphi_{\sigma i}\in \XM$ such that
\begin{align}
	\label{EVPNMK}
	\Hsi(\laxsi) \varphi_{\sigma i} = \laxsi \varphi_{\sigma i}.
\end{align}
 Next, we define the approximate `lifting' operator whose origin will be become clear in the next section:
 \begin{align} \label{Qsig}  
 Q_{\sigma} (\lax)  
 &:= 
 \one - R_{\sigma}(\lax) \PNM V \PM.
\end{align}

\new{
Note that in the case of $N=M$ there holds $\PNM=0$, $\Us(\lax)=0$, and thus, equation~\eqref{EVPNMK} (and ~\eqref{EVPN}) simply reduces to the variational approximation involving the Hamiltonian $\HM$. 
}

\subsection{Main results} 
Within this manuscript, we denote by $\ls$ upper bounds involving constants that do not depend on the parameters $\sigma=(N,M,K), \aaa, r, \| V \|_r$.
Then, we have the following result, \new{whose proof will be provided in Section~\ref{sec:main-proof}.} 
\begin{thm}   
\label{thm:main}
 Let Assumption \ref{as:pot} hold, let $\la$ be an isolated eigenvalue of $\Hcal$ of finite multiplicity $m$, \new{with eigenfunctions $\varphi_{ i}$, and} let $\g_0$ denote the gap \new{between} $\la$ \new{and} the rest of the spectrum of $\Hcal$. 
 
 Then, there exists $\aaa>0$ and $M_0\in\mathbb N$ such that for $\new{N \ge } M \ge M_0$, problem~\eqref{EVPNMK} has $m$ solutions $( \varphi_{\sigma i}, \laxsi)\in \XM\times [\las-\frac{\g_0}{2},\las+\frac{\g_0}{2}]$
 approximating $( \varphi_{ i}, \las)$ in the following sense: 
\begin{align}
    |\la-\laxsi|
    &\ls 
    (\las + \aaa)
    \frac{\Enorm{V}^2 }{\aaa^r}
     \varepsilon(\sigma,r,V),
	\\
   \| \varphi_{i} - Q_{\sigma} (\laxsi)\varphi_{\sigma i} \|
    & \ls 
    \Enorm{V} \left[
    1+ \frac{\lac}{\g_0} \frac{\Enorm{V}}{ \aaa^r} 
    \right]
    \, \varepsilon(\sigma,r,V),
 	\label{eq:est-eigenvectors2-0}
\end{align}
\new{where  $\lac = \la+\g_0+\aaa$ and}
\[
    \varepsilon(\sigma,r,V)
    :=
    \laN^{-r} 
	  + \laM^{-r}  
	 \left[ 4 \laM^{-r}\Enorm{V} \right]^{K+1}.
\]
\end{thm}     
This Theorem is subject to several remarks.
\begin{remark}
Note that $\varepsilon$ is equivalent to 
\[
    \varepsilon(\sigma,r,V)
    \approx 
	 N^{-2r} + M^{-2r}  
	 \left[ 4 \left(\tfrac{L}{2\pi}\right)^{2r}  M^{-2r}\Enorm{V} \right]^{K+1},
\]
where the equivalence constants do not depend on the parameters $\sigma=(N,M,K),r,\aaa,V$.
\end{remark}
\begin{remark}
In some cases, for instance in multi-scale problems, one might be only interested in the coarse-scale solution, i.e. the best-approximation in the coarse space $\XM$ given by $\PM\varphi_{i}$.
In such cases, a useful byproduct of the proof of Theorem~\ref{thm:main} is the following estimate
\begin{equation}
   \| (-\Delta+\id)^s (\PM\varphi_{i} - \varphi_{\sigma i}) \|
    \ls 
    \frac{\lac}{\g_0} \frac{\Enorm{V}^2}{ \aaa^r} 
     \rho_M^s\,
    \varepsilon(\sigma,r,V), 
	\label{eq:est-eigenvectors1-0}
\end{equation}
for any $s\ge 0$, which thus compares the eigenfunctions in the space $\XM$.
\end{remark}
\begin{remark}
Note that convergence of the eigenvalues and the eigenfunctions can be achieved by taking the limit $K,N\to \infty$ for fixed $M\ge M_0$. 
For practical purposes, the idea is to set $N$ large enough so that the error is dominated by the error introduced in $K<+\infty$. 

Further, note that the eigenvalue and eigenvector errors have the same rate of convergence with respect to $K$. However, the error in the eigenvector depends on the gap $\g_0$ while the error in the eigenvalue does not.
\end{remark}

The estimate with respect to $N$ in Theorem~\ref{thm:main} is not sharp in all cases, in particular for sufficiently regular potentials $V$. Nonetheless, our analysis has the merit of presenting the convergence result in one combined analysis based on perturbative techniques which also holds for low regularities of the potential where standard {\it a priori} \new{convergence results of the variational approximation}
\old{approximation results of the variational approximation result} 
do not hold. 
\new{In fact, still in the low regularity regime, an estimate of the variational problem can be obtained by setting $K\to \infty$ or $M=N$.}

Note that we can adapt the result whenever {\it a priori} approximation results are available by employing the triangle inequality.
Indeed, if the potential $V$ belongs to the Sobolev space $V\in\Hs{s}$, with $s>d/2$, we resort to {\it a priori} results in a first place to obtain a sharp bound with respect to $N$, see e.g., \cite{Babuska1991-cg,cances2010numerical,Boffi2010-wh}, and also~\cite{norton2010convergence} for a certain class of discontinuous potentials in $H^{1/2-\varepsilon}$ for all $\varepsilon>0$, in two dimensions.

More precisely, we consider $\Hcal$ acting on $\XN$ directly, i.e. substituting $\Hcal$ by $\Hcal_N:=\PN\Hcal\PN$ and using $\sX=\XN$ with variational solution $(\varphi_{N},\lax_{N})$, assuming a simple eigenvalue for simplicity.
It is important to note that problem~\eqref{EVPNMK} remains unchanged and thus, the result of Theorem~\ref{thm:main} holds with 
\[
    \widetilde\varepsilon(\sigma,r,V)
    :=
	\laM^{-r} 
	 \left[ 4 \laM^{-r}\Enorm{V} \right]^{K+1},
\]
but where the exact solution is substituted by $(\varphi_{N},\lax_{N})$.
Proceeding then by the triangle inequality yields
\begin{align*}
    |\las - \lax_{\sigma i} |
    &\le 
    |\las- \lax_{N}|
    +
    |\lax_{N} - \lax_{\sigma i}|,
    \\
    \|\varphi_i  - Q_{\sigma}(\lax_{\sigma i})\varphi_{\sigma i} \|
    &\le 
    \|\varphi_i - \varphi_{N}\|
    +
    \|\varphi_{N} - Q_{\sigma}(\lax_{\sigma i})\varphi_{\sigma i}\|.
\end{align*}
Combining then the aforementioned {\it a priori} estimates from~\cite{Babuska1991-cg,cances2010numerical,Boffi2010-wh}
for the first terms of the right hand sides with Theorem~\ref{thm:main} for the latter parts yields the following corollary.

\begin{cor}
\label{rem:TriangleIneq}
   Under the conditions of Theorem~\ref{thm:main} and if $V\in\Hs{s}$, $s>d/2$, then
\begin{align*}
  |\las - \lax_{\sigma i} | 
  &\le  
  C \, \left( N^{-(2s+2)} + \widetilde\varepsilon(\sigma,r,V)
    \right),
    \\
  \|\varphi_i - Q_{\sigma}(\lax_{\sigma i})\varphi_{\sigma i}\|
  &\le  
  C \, \left( N^{-(s+2)} + \widetilde\varepsilon(\sigma,r,V)
    \right),
\end{align*}
for some constant $C>0$ independent on $\sigma=(N,M,K)$.
\end{cor}

\begin{remark}
    \label{rem:FSM-otherbasis} \new{Instead of the space $\XM$ defined in \eqref{XM}  spanned by the eigenfunctions of $-\Delta$ on $\cR$, 
    we could have taken another finite dimensional approximation of the space $L^2(\R)$. }
\end{remark}
\begin{remark}
    \label{rem:FSM-multigrid} \new{Given a sequence of larger and larger finite dimensional spaces, one could apply the FS maps consequently with larger and larger projections. 
    In particular, this could lead to a version of a multigrid technique.} 
\end{remark}

\subsection{Theoretical background}\label{sec:theor}
Let us shed \old{the attention to}\new{light on} the theoretical foundation on the eigenvalue formulation in form of~\eqref{EVPN}, instead of~\eqref{eq:EVP}. \new{For a pair of 
projections $P$ and $P^\perp$  such that $P+P^\perp = \id$, we define the set $D_P$ of operators $H$ on a Hilbert space  such that $H^\perp:=P^\perp H P^\perp$  is invertible 
on $\Ran P^\perp$ and the operators $HP$ and $PH$ are bounded. Furthermore, for an operator $H\in D_P$, we define the bounded operator 
\begin{equation} \label{Fesh}
 F_{P} (H) \ := \ P (H  - H R^\perp H) P ,
\end{equation}
where $R^\perp:=P^\perp  (H^\perp)^{-1} P^\perp$,  acting on the subspace $\Ran P$. 
Our approach} is based on the following results, originally presented in~\cite[Theorem 11.1]{Gustafson2011-wd}. 
\begin{thm}
\label{thm:isospF}
Let $P$ and $P^\perp$ be a pair of 
projections such that $P+P^\perp = \id$.  
Then $F_{P}\new{:H \ra F_{P} (H ) }$, considered as  a map \new{from the subspace $D_P$} into the space of bounded operators \new{acting on the subspace} \new{$\Ran P$,}  is \textit{isospectral} in the following sense: \new{if an operator $H$ and a number $\lambda\in \R$ are such that $H-\lambda\in D_P$, then}
\begin{itemize}
\item[(a)] $ \lambda \in \sigma(H ) \qquad \Longleftrightarrow \qquad 0 \in \sigma(F_{P}
(H - \lambda))$; 
\item[(b)]  
$H\psi = \lambda \psi \qquad \Longleftrightarrow\qquad 
F_{P} (H - \lambda) \, \varphi = 0;$
\item[(c)] 
$\dim Null (H - \lambda) = \dim Null F_{P} (H -
\lambda)$.
\end{itemize}
Moreover, $\psi$ and $\varphi$ in (b) are related as $\varphi= P\psi$ and $\psi=Q_{P}(\lax)\varphi$, where 
\[
    Q_{P}(\lax)   :=  P - P^\perp \, (H^\perp- \lambda)^{-1}P^\perp H P.
\]
Finally, if $H$ is self-adjoint, then so is $F_{P} (H )$.
\end{thm}
\new{Here, $Null$ denotes the null space (or kernel) of the following operator.}
The map $F_{P}$ on the space of operators, is called the {\it Feshbach--Schur map}. The relation $\psi=Q_{P}(\lax)\varphi$ allows us to reconstruct the full eigenfunction from the projected one.
By statement~(a), we have
\begin{cor}\label{cor:nuFP} 
Let $\nu_i(\lax)$ denote the $i$-th eigenvalue of the operator $F_{P} (H - \lambda)+\lambda $ for each $\lax$ in an interval $I\subset \R$. 
\old{Then the eigenvalues  of $H$ in  $I$ are in one-to-one correspondence with the solutions of the equations}
\new{Then, there exists a bijection between the eigenvalues of $H$ in $I$ and the solutions of the equation}
\[
    \nu_i(\lax)=\lax.
\] 
\end{cor}

In the current setting of \old{planewave}\new{the spectral Fourier} approximations, 
  $P=\PM$, $P^\perp = \PMp$,    Proposition \ref{prop:UN-well-defined} implies 
   that the results of Theorem~\ref{thm:isospF} apply for each choice of $M\in\mathbb N$  and yield\begin{align}
	\label{FS-HNlam}
 F_{\PM} (\Hcal - \lambda)=   \HMl- \lambda\PM, 
\end{align}
where we introduced  the notation 
\begin{align}
	\label{HNlam-def}
    \HMl:= \HM + \UM(\lam).
\end{align}
Note that $\HMl$ is exactly the operator entering \eqref{EVPN}. Thus, we have the following. 

\begin{cor}
\label{cor:isospFN} 
Let $\lax \in \C$ with 
 $\re\lax <  \ka_M$. Then
\begin{itemize}
\item[(a)]  
$\Hcal \psi = \lambda \psi \quad \Longleftrightarrow\quad 
(\HMl - \lambda) \, \varphi_M = 0;$
\item[(b)] 
$\dim Null (\Hcal - \lambda) = \dim Null  (\HMl - \lambda)$.
\item[(c)] $\psi$ and $\varphi$ in (a) are related as $\varphi_M= \PM\psi$ and $\psi=\QM(\lax)\varphi_M$, where
\begin{align}
    \label{eq:QM}
    \QM(\lax) 
    &=  
    \one - (\HMp -\lambda)^{-1} \PMp V \PM.
\end{align}
i.e. the corresponding eigenfunction can be reconstructed from $\varphi_M$ by an explicit linear map.
\end{itemize}
\end{cor}
  This result shows that the original infinite-dimensional spectral problem~\eqref{eq:EVP} is equivalent to the finite dimensional spectral problem~\eqref{EVPN} which is nonlinear in the spectral parameter $\lax$. 
  We now state a few properties of the effective interaction $\UM(\lax)$, in order to characterize the solutions of the fixed-point problems $\nu_i(\lax) = \lax$. \new{First, we give a definition. A family of bounded, self-adjoint operators $T(\lambda):X\to X$, with $\lambda\in I\subset\mathbb R$, is said to be monotonically decreasing if $T(\lambda) < T(\lambda')$, in the sense of quadratic forms (i.e. $\lan u, T(\lambda)u\ran < \lan u, T(\lambda')u\ran$), for all $\lambda>\lambda'$. 
  Similarly, one can define increasing, non-increasing, and non-decreasing families of operators.
  If $T(\lambda)$ is a weakly differentiable family, then, by the fundamental theorem of calculus, if $T'(\lambda)\le 0$ and is not identically $0$, then $T(\lambda)$ is monotonically decreasing. }
 
\begin{prop}
\label{prop:UN-prop}
 For $\lax\in \R$ such that
$\lax <  \ka_M$, %
$\UM(\lax)$ is 
(i) non-positive, 
 (ii) monotonically decreasing with $\lam$, (iii) vanishing as $\lax\ra-\infty$. For $\lax \in \C$ such that $\re\lax <  \ka_M$, $\UM(\lax)$ is (iv) complex analytic in $\lax$ and (v) 
 symmetric.  
\end{prop}
 \begin{proof} 
 Properties (i)-(iv) follow directly from definition  \eqref{UNlam-def} 
  and Lemma \ref{lem:Hperp-low-bnd} above. 
For the last one, we use that $\Hcal$ is self-adjoint.
\end{proof}
 
\begin{prop}
\label{prop:nu-fp-sol} 
Denote by $\nuMi(\lax)$ the $i$-th eigenvalue of $\HMl$
\new{ and assume that the $i$-th eigenvalue of $\Hcal$ is less than $\ka_M$.
Then, the equation $\nuMi(\lax)=\lax$ has a unique solution in the interval $(-\infty,\ka_M)$.
}
\old{. Then the equation $\nuMi(\lax)=\lax$ has a unique solution on the interval $(-\infty,\ka_M)$.}
\end{prop}
 \begin{proof} Since, for $\lax <  \ka_M$, $\UM(\lax)$ is symmetric, the operator $\HMl$ defined by \eqref{HNlam-def} is (a) self-adjoint, (b) monotonically decreasing with $\lam$, (c) converging to $\HM $ as $\lax\ra-\infty$, (d) is complex analytic in $\lax$ for $\re\lax <  \ka_M$.
We deduce from (b) that the functions $\nuMi$ are decreasing on $(-\infty,\ka_M)$ and thus, if the $i$-th eigenvalue of $\Hcal$ is less than $\ka_M$, the equation $\nuMi(\lax)=\lax$ has a unique solution \new{(that equals the $i$-th eigenvalue of $\Hcal$).}
\end{proof}
Note also that $\lim_{\lax \rightarrow -\infty} \nuMi(\lax)$ is the $i$-th eigenvalue of $\HM$ which is larger than the $i$-th eigenvalue of $\Hcal$ due to the variational principle.

These considerations motivate the numerical strategies to compute solutions to~\eqref{EVPNMK} in the following section.

\subsection{Numerical strategy}
In order to find solutions to the non-linear eigenvalue problem~\eqref{EVPNMK}, we propose two strategies:

\textit{Strategy 1:} 
For a fixed index $i=1,\ldots,M$, consider the sequence of iterates $\lax_{\sigma }^{(j)}$ obtained by 
\begin{align}
    \label{EVPNMKk}
	\lax_{\sigma }^{(j)}: \text{ is the $i$-th eigenvalue of }\Hsi(\laxs^{(j-1)}).
\end{align}
We thus introduce the notation $\nusi(\lax)$ denoting the $i$-th eigenvalue (counting multiplicities) of the Hamiltonian $\Hsi(\lax)$ and thus have $\lax_{\sigma }^{(j)}= \nusi(\laxs^{(j-1)})$.  
The limit value $\lax_{\sigma }:=\lim_{j\to \infty}\lax_{\sigma }^{(j)}$ then satisfies $\lax_{\sigma } = \nusi(\laxs)$ and thus~\eqref{EVPNMK}. 

\textit{Strategy 2:}
For a given target value $\lat \in \R$, consider the sequence of iterates $\lax_{\sigma }^{(j)}$ obtained by 
\begin{align}
    \label{EVPNMKk-2}
	\lax_{\sigma }^{(j)}: \text{ is the eigenvalue of }\Hsi(\laxs^{(j-1)})\text{ closest to }\lat.
\end{align}
We thus introduce the notation $\nu_{\sigma {\sf t}}(\lax)$ denoting the eigenvalue of the Hamiltonian $\Hsi(\lax)$ closest to $\lat$ and thus have $\lax_{\sigma }^{(j)}= \nu_{\sigma {\sf t}}(\laxs^{(j-1)})$.  
The limit value $\lax_{\sigma }:=\lim_{j\to \infty}\lax_{\sigma }^{(j)}$ then satisfies $\lax_{\sigma } = \nu_{\sigma {\sf t}}(\laxs)$ and thus~\eqref{EVPNMK}. 

\medskip

In both cases, as outlined in the upcoming Remark~\ref{rem:ConvFP}, convergence of the fixed-point maps~\eqref{EVPNMKk} and \eqref{EVPNMKk-2} can be guaranteed under some conditions and for \new{$N,M$} large enough. 

\begin{remark}
    \new{This numerical strategy can be easily extended for eigenvalues with multiplicity higher than one. Indeed, using Proposition~\ref{prop:nu-fp-sol}, provided that the sought eigenvalue is less than $\ka_M$, one can iteratively compute a given number of the lowest eigenvalues of the operator $\Hsi({\laxs})$, with ${\laxs}$ updated at each iteration, matching a close value of the multiple eigenvalue. At convergence, this eigenvalue and the multiplicity will match the exact one, up to an error given in Theorem~\ref{thm:main}.
    Regarding bands of eigenvalues, one could as well iteratively compute a given number of the lowest eigenvalues of the operators $\Hsi({\laxs}_i)$, with the ${\laxs}_i$ updated at each iteration, but this multiplies the number of problems that have to be solved by the number of eigenvalues in the band. To do it in an efficient manner would require to 
    consider a density-matrix formalism, in order to consider all interesting eigenvectors together, but this goes beyond the scope of the present paper.
    }
\end{remark}
\blue
\noindent
The complexity of the numerical method can be summarized as follows:
\begin{itemize}
\item[(1)] We assume that the resolution of the non-linearity with the iterative scheme \eqref{EVPNMKk} or \eqref{EVPNMKk-2} requires $J$ iterations.
\item[(2)] At each iteration $j=1,\ldots,J$, one requires the resolution of an eigenvalue problem with $O(M^d)$ degrees of freedom and we assume that this problem is solved with an iterative solver using $J_{\rm evp}^M$ matrix-vector multiplications.
\item[(3)] The computation of one matrix-vector product corresponding to the Hamiltonian $\HM + \UM(\lam)= -\PM \Delta\PM + \PM V \PM + \UM(\lam) $ contains three parts. \\
(i) the matrix corresponding to $\PM \Delta\PM$ is diagonal in the Fourier basis; (ii) the application of $\PM V \PM$ can be effected using the Fast Fourier Transform (FFT) on the  grid $\mathsf X_{M}$ and scales as $M^d\log(M^d)$;\\
(iii) the matrix-vector product corresponding with the application of the multiplicative potential $U_\sigma(\lambda)$.
For convenience we recall here its definition:
\[
	\Us(\lax)
	=
	-\PM V \PNM
	 \left(\sum_{k=0}^K (-1)^{k}\Big[ \GNM(\lax)\VMN \Big]^k \GNM(\lax)\right) \PNM V \PM,
\]
with $\GNM(\lax) = ( -\Delta - \lax)|_{\Ran \PNM}^{-1}$. 
The application of $\GNM(\lax)$ scales as $N^d-M^d$ and the application of the potentials $\PM V \PNM$, $\VMN$ and $\PNM V \PM$ can be performed again using the FFT on the entire grid $\mathsf X_{N}$ and it scales as $N^d \log(N^d)$.
\end{itemize}
The overall complexity is thus proportional to
\[
    J \times J_{\rm evp}^M \times \Big( M^d + M^d \log(M^d) + (K+3) N^d \log(N^d) + (K+2) (N^d-M^d) \Big)
\]
In contrast, following the same notation, the variational problem requires a complexity of 
\[
    J_{\rm evp}^N \times  N^d \log(N^d)
\]
Therefore in order to compare the methods,  it depends on $J \times J_{\rm evp}^M(K+3)$ versus $J_{\rm evp}^N$ at comparable accuracy.

Finally, we emphasize that the focus of the present paper is to present the numerical analysis of this new method, which covers cases where the analysis of the standard variational approximation does not hold and efficiency becomes a secondary factor.
\black

\section{Perturbation estimates}\label{sec:pert-est}
In this article, we often deal with the following eigenvalue  perturbation problem: 
Given an operator $H$ on a Hilbert space $X$ of the form
\begin{equation}
\label{eq:pertform}
  H = H_0 +  W,
\end{equation}
where $H_0$ is an operator with some isolated eigenvalues and $W$ is small in an appropriate norm,  show that $H$ has eigenvalues near those of $H_0$ and estimate these eigenvalues and the corresponding eigenvectors.
We therefore start by presenting an abstract theory which will be applied to our concrete problem in the following sections.

Specifically, we assume that $H$ and $H_0$ are self-adjoint and bounded from below and that $W$ is 
{\it $\aaa$-form-bounded} w.r.t. of $H_0$, in the sense that for  $\aaa\in\R$ such that $H_0+\aaa$ is a positive operator ($H_0 +\aaa> 0$), we have
\begin{align} \label{H0a-norm} 
\| W  \|_{\! H_0,\alpha} := \|(H_0+\alpha)^{-1/2} W (H_0+\alpha)^{-1/2}\|<\infty,
\end{align}
where 
$(H_0+\aaa)^{-s}, s>0,$ is defined either by the spectral theory or by the explicit formula \[(H_0+\aaa)^{-s}:=c_\aaa \int_0^\infty (H_0+\aaa+\om)^{-1} d\om/\om^s,\] where $c_\aaa:=[\int_0^\infty (\aaa+\om)^{-1} d\om/\om^s]^{-1}$. 
This notion is equivalent to that of the relative form-boundedness, but it gives an important quantification of the latter. 

We also note here that, by a known result about relatively form-bounded operators (see e.g.~\cite{RSII, HS}),  if $H_0$ is a self-adjoint, bounded below operator on $X$ and  $W$ is symmetric and  $\aaa$-form-bounded w.r.t. of $H_0$, then  $H = H_0 + W$ is self-adjoint.

We start with a general result on the eigenvalue difference.
\begin{prop}
\label{prop:eigenvalue-estimate}
    Let $H_0$ be a self-adjoint bounded below operator on $X$ and  $W$  symmetric and 
 $\aaa$-form-bounded w.r.t. of $H_0$, and let  $H = H_0 + W$.
Let $\alpha\in\R$ be such that  $H_0+\aaa>0$.
 Then \old{$H$} the eigenvalues of $H$ and $H_0$ satisfy the estimates 
\begin{align}
    \label{lami-lam0-est}
     |\nu_{i}(H) - \nu_i(H_0)| \, &\le (\nu_i(H_0) + \alpha) \| W  \|_{\! H_0,\aaa},
\end{align}
where $\nu_{ i}(A)$ denotes the $i$-th eigenvalue of the operator $A$.
\end{prop}

\begin{proof}
Let $u\in X$ be arbitrary and define $v=\Gza^{1/2}u$ noting that $\Gza >0$.
Then, 
\begin{align*}
    \langle u, H u \rangle
    &=
     \langle u, H_{0}  u\rangle
    +  \langle v, \Gza^{-1/2} W \Gza^{-1/2} v\rangle.
 \end{align*}
Note that
\[
    \langle v, \Gza^{-1/2} W \Gza^{-1/2} v\rangle
    \le \| W  \|_{\! H_0,\aaa} \langle v, v \rangle
    = \| W  \|_{\! H_0,\aaa} \langle u, \Gza u \rangle,
\]
and therefore
\begin{align*}
     \langle u, H_0  u\rangle \left(1  - \| W  \|_{\!H_0,\aaa}\right)& - \alpha \| u\|^2\| W  \|_{\!H_0,\aaa}
     \le 
   \langle u, H u\rangle\\
   &\le 
     \langle u, H_0  u\rangle \left(1 + \| W  \|_{\!H_0,\aaa} \right) + \alpha \| u\|^2\| W  \|_{\!H_0,\aaa}.
\end{align*}
Using the min-max principle (Courant--Fisher), there holds
\[
    \nu_i(H_0) \left(1  - \| W  \|_{\!H_0,\aaa}\right) - \alpha\| W  \|_{\!H_0,\aaa}
    \le 
    \nu_i(H)
    \le 
    \nu_i(H_0) \left(1 + \| W  \|_{\!H_0,\aaa} \right) +\alpha \| W  \|_{\!H_0,\aaa},
\]
which leads to the result.
\end{proof}

Let us now assume that $\lam_{0 }$ is an isolated eigenvalue of $H_0$ of finite multiplicity $m$ and
let $\Pz$ be the orthogonal projection onto the span of the the eigenfunctions 
of $H_0$ corresponding to the eigenvalue $\lam_{0 }$, and let $\oP := \one - \Pz$. 
We further introduce $\Hzl:=H_0^\perp-\lax \oP$ and thus $\Hza=H_0^\perp+\aaa\oP$.

Let $\g_0$ denote the gap of $\lax_0$ to its closest eigenvalue in the remaining spectrum of $H_0$ and we introduce the spectral interval $\I= [\lax_{0 }- \frac12 \g_0,  \lax_{0 }+ \frac12 \g_0 ] $.
 
 Our next result gives estimates on the difference of eigenvectors of $H$ and $H_0$, as well as on the difference of their corresponding eigenvalues.  
For standard approaches to the spectral perturbation theory, see \cite{Re, Kato1976-hm, RSIV, HS}.

\begin{thm}  \label{thm:FS-pert-var}
  Let $H_0$ be a self-adjoint bounded below operator on $X$, with the eigenvalue $\lam_{0 }$ as above, and  $W$  symmetric and 
 $\aaa$-form-bounded w.r.t. $H_0$, and let  $H = H_0 + W$.
 Let $\alpha\in\R$ be such that  $H_0+\aaa>0$.
If $ \| W  \|_{\!H_0,\aaa} \le \frac{1}{2} \frac{\g_0}{\lax_0+\aaa}$, then
the self-adjoint operator $H$ has  exactly $m$ eigenvalues (counting the multiplicities), denoted by $\mu_i$, in the interval $\I= [\lax_{0 }- \frac12 \g_0,  \lax_{0 }+ \frac12 \g_0 ]$ which satisfy
\begin{equation}
\label{eq:eigenvalue_est}
    |\mu_{ i}-\lax_{0 } | 
    \le  
    (\lax_{0 }+ \aaa)  \|W \|_{H_0, \aaa}
    \le 
    \frac12 \g_0.
\end{equation}
Further, if $ \| W  \|_{\!H_0,\aaa} \le \frac{1}{4} \frac{\g_0}{\lac}$, then  any normalized eigenfunction, $\psi_i$, of $H$ for the eigenvalue $\mu_i$ satisfies the estimates 
\begin{align}
   \| \Hza^{1/2} (\psi_{0i} - \psi_i) \|
    & \le 4 \frac{\lac}{\g_0} (\lax_0+\alpha)^{1/2}  \| W   \|_{\!H_0,\aaa}, \label{eq:est-eigenvectors1}
    \\
   \| \psi_{0i} - \psi_i \|
    & \le 4 \frac{\lac}{\g_0} \| W   \|_{\!H_0,\aaa}, \label{eq:est-eigenvectors2}
\end{align}
where $\lac=\lax_0 + \alpha+\g_0$ and $\psi_{0i}$ is an appropriate eigenfunction of $H_0$ corresponding to the eigenvalue $\lam_0$, namely $\psi_{0i} := \Pz \psi_i$. 
\end{thm}

\begin{remark}
    We note that similar estimates can be obtained for normalized eigenfunctions $\widetilde\psi_{0i} := \Pz \psi_i/\|\Pz \psi_i\|$ with an additional factor 2 using the estimate
    \begin{align*}
        \|\widetilde\psi_{0i} - \psi_i \|
        &\le  \big|1-\|\psi_{0i}\|\big| + \|\psi_{0i} - \psi_i \|
        \le  \big|\|\psi_i\|-\|\psi_{0i}\|\big| + \|\psi_{0i} - \psi_i \|
        \le 2\, \|\psi_{0i} - \psi_i \|.
    \end{align*}
 \end{remark}

We first develop the following preliminary results.  

\begin{lem}
\label{lem:H-lam-alph2}
Let $\alpha\in\R$ be such that  $H_0+\aaa>0$ and $\lac:=\lax_{0 }+\g_{0 }+ \aaa$.
Then, for all $\lax\in I_0^c:=\{z\in \C: \re z\in I_0\}$,
there holds $|\oP - (\lax+\alpha)  \Hza^{-1} \oP| \ge \frac{\gam_0}{2\lac}$.
\end{lem}
\begin{proof}
The eigenvalues of $|\oP - (\lax+\alpha)  \Hza^{-1} \oP|$ on $\Ran \oP$ are 
\[
\left| 1 - \frac{\lax+\alpha}{\lam_{0i}+\alpha} \right|
= 
\left|\frac{\lam_{0i} - \lax}{\lam_{0i}+ \alpha}\right|,
\]
where $\lam_{0i}$ denotes the eigenvalues of $H_0$ and the index $i$ runs over all eigenvalues except $i$ such that $\lam_{0i} = \lam_0$. 
For $\lax\in I_0^c$, we write $\lax=\lax_r + \mbox{i} \lax_i$, with $\lax_r\in \I$, $\lax_i\in\mathbb R$.
Since for any $x\in\mathbb R$, $|x-\lax| \ge |x-\lax_r|$ we have thus to study the function 
 \[
    f(x) = \left|\frac{x-\lax}{x+\alpha} \right|, \qquad x\in K_\alpha := [-\alpha,+\infty)
    \setminus (\lax_0- \gamma_0,\lax_0+ \gamma_0),
 \]
 for $\lax \in \I$ in order to lower bound the eigenvalues.
\noindent 
 Since 
 \[
    f'(x) = \frac{x-\lax}{|x-\lax|}\cdot\frac{\aaa+\lax}{(x+\alpha)^2},
\]
if $\alpha + \lambda \le 0$, 
 there holds
$
    f'(x) < 0 \ \mbox{for } x>-\alpha
$
so that
 \[
    \min_{x\in K_\alpha} f(x) \ge 1.
 \]
 If $\aaa+\lax>0,$
 there holds
\[
    f'(x) < 0 \ \mbox{for } x<\lax,\qquad\qquad
    f'(x) > 0 \ \mbox{for } x>\lax,
\]
and thus, for $\lax\in I_0$,
\[
    \min_{x\in K_\alpha} f(x) 
    = \min\left(  f(\lax_0-{\g_0}), f(\lax_0+{\g_0}) \right)
    = \min \left( \frac{|\lax_0-\g_0-\lax|}{\lax_0-\g_0+\aaa},
    \frac{|\lax_0+\g_0-\lax|}{\lax_0+\g_0+\aaa} \right)
    \ge \frac12 \,\frac{\g_0}{\lac},
 \]
yielding the result.
\end{proof}

Denote  $H^\bot := \oP H \oP \restriction_{\Ran \oP}$ and $R^\perp(\lam):= \oP (H^\bot-\lambda)^{-1} \oP$. We have

\begin{lem}\label{lem:FSM-conds}
Let $\alpha\in\R$ be such that  $H_0+\aaa>0$ and $\lac:=\lax_{0 }+\g_{0 }+ \aaa$.
Let $ I_0^c:=\{z\in \C: \re z\in I_0\}$ and assume $\|W\|_{H_{0, \aaa}}\le \frac{1}{4}\frac{\g_0}{\lac}$. 
Then, for $\lam\in I_0^c$, the following statements hold
\begin{itemize} 
\item[(a)]  
The operator $H^\bot-\lambda$   is invertible on $\Ran \oP$;
\item[(b)]  
 The inverse $R^\perp(\lam):=\oP(H^\perp- \lax)^{-1}\oP$ defines a bounded, analytic operator-family;

\item[(c)] 
The expression \begin{equation}\label{Ulam}
U(\lambda):=- \Pz H R^\perp(\lam) H \Pz
\end{equation}
 defines a finite-rank, analytic operator-family and bounded as
\begin{align}\label{U-bnd}
	&\|U(\lax)\|_{H_{0, \aaa}}
	\le  
	4\,[\lac/\g_0]\|\Pz W \oP\|_{H_{0, \aaa}}^2
	\le  
	4\,[\lac/\g_0]\|W\|_{H_{0, \aaa}}^2. 
\end{align}
Further, $U(\lax)$ is symmetric for any $\lam\in I_0$.
\end{itemize} \end{lem}

\begin{proof} (a)  Since $H^\perp$ is self-adjoint, the operator $H^\perp - \lax$ is invertible for any $ \lax\in\C\backslash\R$. For $\lax\in I_0$, we argue as follows.    With the notation $A^\bot := \oP A\oP \restriction_{\Ran \oP}$, we write
\[
  H^\perp = H_{0}^\perp +  W^\perp.
\]
 Now, we write
 \begin{equation}
 \label{eq:35000}
        H^\perp- \lax \oP =\Hza^{1/2} [ \oP - (\lax+\alpha)  \Hza^{-1} + K_{\lam}] \Hza^{1/2},
 \end{equation}
 with $K_{\lam} = \Hza^{-1/2} W^\perp \Hza^{-1/2}$.
Lemma~\ref{lem:H-lam-alph2} yields that $| \oP - (\lax+\alpha)  \Hza^{-1} P_0^\perp | \ge \frac{\gam_0}{2\lac}$ and thus, the operator $\one - (\lax+\alpha)  \Hza^{-1} + K_{\lam}$ is invertible as soon as $\| K_\lax \| < \frac{\gam_0}{2\lac},$ which is in particular the case if $\|W\|_{H_0, \aaa}\le \frac{1}{4}\frac{\g_0}{\lac}$.
Then, we also have $\left| \one - (\lax+\alpha)  \Hza^{-1} + K_{\lam} \right| \ge \frac{1}{4}\frac{\g_0}{\lac}.$ Hence the operator $H^\perp- \lax$ is a product of three invertible operators and therefore is invertible itself on $\Ran \oP$.

For (b), since $H^\perp- \lax$ is invertible on $\Ran \oP$,  
  the interval $I_0$ is contained in the resolvent set, $\rho(H^\perp|_{\Ran\oP})$, of $H^\perp|_{\Ran\oP}$ and therefore, since $H^\perp|_{\Ran\oP}$ is self-adjoint, 
 $I_0^c\subset\rho(H^\perp|_{\Ran\oP})$.
Since  
\begin{equation}
    \label{resolv-id2}
    R^\perp(\lam):=\oP(H^\perp- \lax)^{-1}\oP
\end{equation}
is the resolvent of the operator $H^\perp- \lax$ restricted to $\Ran\oP$, it is analytic on its resolvent set and in particular on $I_0^c$. 

To prove statement (c), we note that the operators $R^\perp(\lam), H \Pz$ and  $\Pz H= (H\Pz)^*$ are bounded and  $R^\perp(\lam)$ is symmetric for $\lam\in I_0$. Hence so is $U(\lam)$. The analyticity of $U(\lam)$ follows from the analyticity of $R^\perp(\lam)$. 
It it clear that $U(\lax)$ is of finite rank due to its definition.

Finally, to prove estimate \eqref{U-bnd}, we first show that
the operator $R^\perp(\lam):=\oP(H^\perp- \lax)^{-1}\oP$, $\lax\in I_0^c$, satisfies
\begin{align}
	\label{ResPerp-bnd2}
	\|\Hza^{1/2} R^\perp(\lam)\Hza^{1/2}\|
	\le 4\, \frac{\lac}{\g_0}.
% 	\ls ( \aaa+\lax_0)/\g_0.  
\end{align}
To this end, we invert \eqref{eq:35000} on $\Ran\oP$ and use that
$\left| \one - (\lax+\alpha)  \Hza^{-1} + K_{\lam} \right| \ge \frac{1}{4}\frac{\g_0}{\lac}$ to obtain~\eqref{ResPerp-bnd2} for  $\lax\in I_0^c$.

Finally, we prove inequality~\eqref{U-bnd}. Since $\Pz H_0 = H_0 \Pz$ 
  and $\Pz \oP = 0$, we have
\[
  \Pz H \oP =  \Pz W \oP, \quad
  \oP H \Pz = \oP W \Pz.
\]
These relations and definition \eqref{Ulam} yield
\begin{align} 
  \label{U-expr} &U(\lam) = - \Pz W R^\perp(\lam) W \Pz.
\end{align}
Combining~\eqref{ResPerp-bnd2} and~\eqref{U-expr}, we obtain~\eqref{U-bnd}. 
\end{proof}	

\begin{remark}
	 In addition, we have the estimate
\begin{align}\label{U-bnd2}	&\|U(\lax)\|\le 4 \, \frac{\lac^2}{\g_0}\|\Pz W \oP\|_{H_0, \aaa}^2 .
\end{align}
Indeed, since 
	$H_0 \Pz=\lax_{0 } \Pz,$ 
	 we have
$
\label{P-est}\|(H_0+\aaa)^{1/2} \Pz\|^2 = \lax_0+\aaa,
$
which implies the estimate
\begin{align}
	\label{PAP-est}
	&\|\Pz A \Pz\| =  (\lax_0+\aaa)\|\Pz A \Pz\|_{H_{0, \aaa}} ,
\end{align}
which, together with  estimate \eqref{U-bnd}, yields   \eqref{U-bnd2}. 
\end{remark}
 Hence, under the conditions of Lemma~\ref{lem:FSM-conds} and for $\lax\in\I$ the following Hamiltonian is well defined

\begin{equation}
    \label{Hlam-1}
    H(\lax) := \Pz H \Pz + U(\lax).
\end{equation}
Note that $\Pz H \Pz=\lax_0\Pz$. Lemma \ref{lem:FSM-conds} above implies
\begin{cor}\label{cor:U-prop} 
The operator family
$H(\lax)$ is 
(i)  self-adjoint for  $\lax\in I_0$ and 
(ii)  complex analytic in $\lax \in I_0^c$.   \end{cor}

In what follows, we label the eigenvalue families $\nu_i(\lax)$,  $i=1,\ldots,m$, of $H(\lax)$ in the order of their increase and so that 

\begin{align}
    \label{nui-order}
    \nu_1(\lax) \le \ldots \le \nu_{m}(\lax).
\end{align}

Note that the eigenvalue branches $\nu_i(\lax)$ can also be of higher multiplicity.
On a subinterval $I_i\subset I_0$, we say that the \emph{branch $\nu_i(\lax)$ is isolated on $I_i$} if each other branch $\nu_j(\lax)$, with $\lax \in I_i$, either i) coincides with $\nu_i(\lax)$ or ii) satisfies
\begin{equation}
    \label{eq:BranchCond}    
    \min_{\lax\in I_i} |\nu_i(\lax) - \nu_j(\lax)|\ge \g_i >0.
\end{equation}

\begin{figure}[t!]
	\centering
    \includegraphics[trim = 0mm 140mm 50mm 0mm, clip,width=0.45\textwidth]{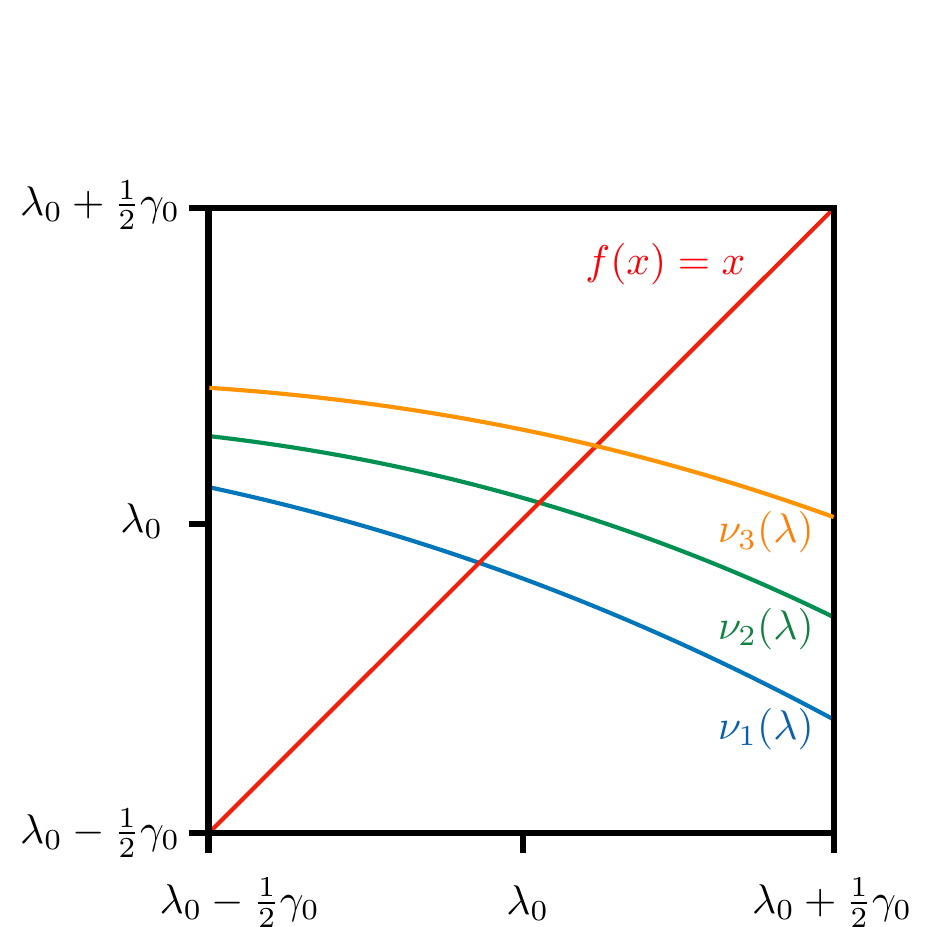}
	\includegraphics[trim = 0mm 140mm 50mm 0mm, clip,width=0.45\textwidth]{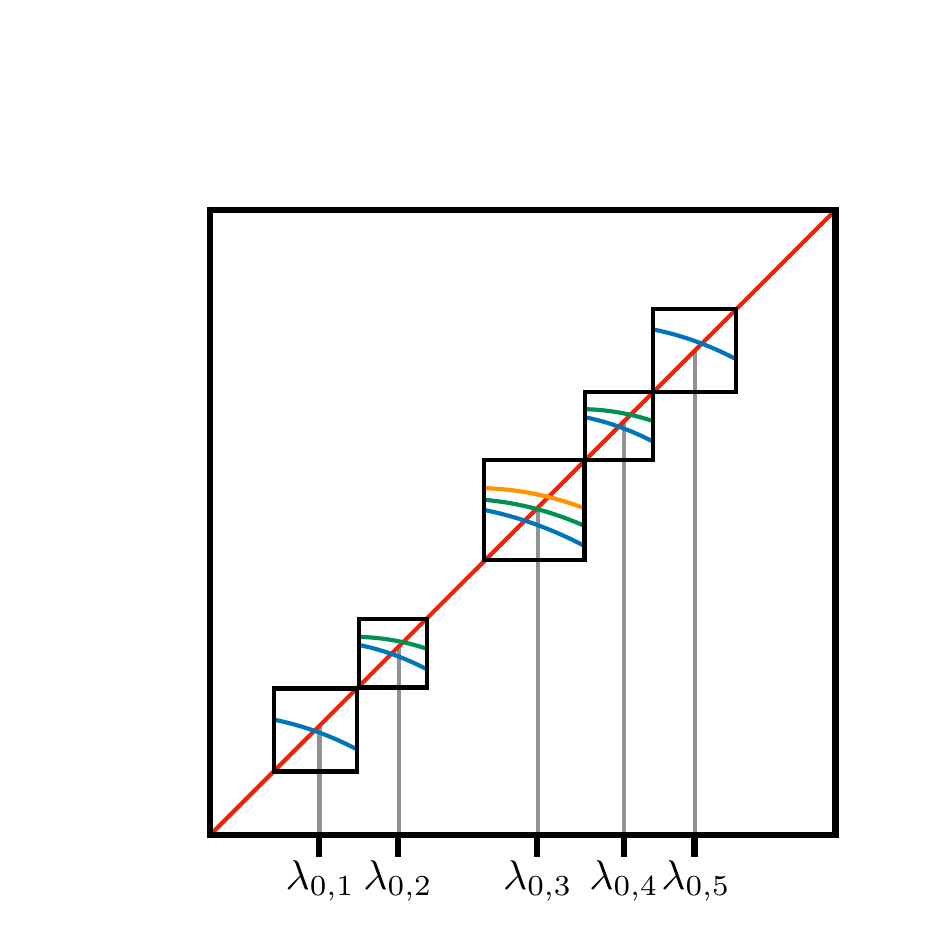}
    \caption{(Left) Schematic illustration of the eigenvalues $\nu_i(\lax)$ of $H(\lax)$ in the neighborhood of $\lax_0$ for the case of $m=3$.
        (Right) Illustration of the spectrum of $H_0$ consisting of five eigenvalues $\lax_{0,1}\ldots,\lax_{0,5}$ of multiplicity $m_{0,1}=1$, $m_{0,2}=2$, $m_{0,3}=4$, $m_{0,4}=2$, $m_{0,5}=1$ and the corresponding situation when zooming in close to $\lax_0 = \lax_{0,i}$.
 	}
	\label{fig:zoom}
\end{figure}

Further, we have the following result.

 \begin{prop}\label{prop:H-evs} 
Let $\alpha\in\R$ be such that  $H_0+\aaa>0$ and let
  $I_i\subset I_0$ be such that the branch $\nu_i(\lax)$ is isolated on $I_i$.
 For $\lax\in I_i$, 
 (i) the eigenvalues $\nu_i(\lax)$ of $H(\lax)$  are continuously differentiable; 
 (ii) the derivative $\nu_i'(\lax)$ is non-positive; 
 (iii) the solutions to the equations $\nu_i(\lax) = \lax$ are unique if $\lax\in I_i$; 
 (iv) if  $\|W\|_{H_{0, \aaa}}\le \frac{1}{4}\frac{\g_0}{\lac}$, the derivatives $\nu_i'(\lax)$, $\lax\in\I':= [\lax_{0 }- \frac14 \g_0,  \lax_{0 }+ \frac14 \g_0 ]\cap I_i $, 
are bounded as 
\[
    |\nu_i'(\lax)| \le   \frac{8}{\pi} \frac{(\lax_0 +\alpha)\lac}{\g_0^2} \|P_0 W \oP\|_{H_0, \aaa}^2.
\]
where $\lac:=\lax_{0 }+\g_{0 }+ \aaa$.
\end{prop}
\begin{proof}
Proof of (i) of a simple eigenvalue $\lax_0$, i.e., $m=1$.
In such a case, $\Pz$ is a rank-one projector on the space spanned by the eigenvector $\varphi_0$ of $H_0$ corresponding to the eigenvalue $\lax_0$ and therefore Eq. \eqref{Hlam-1} implies that $H(\lam) =\nu_1 (\lax) \Pz$, with
\begin{align}\label{FS-sim}
  &\nu_1 (\lax): = \lan \varphi_{0}, H(\lam) \varphi_{0}\ran. 
\end{align}
This and Corollary \ref{cor:U-prop} show that the eigenvalue $\nu_1 (\lax)$ is  analytic.

We now prove (i) in the general case. 
First, we claim the following well-known formula 
\begin{align}
	\label{nu-deriv}{\nu}_i'(\lax)=\lan \chi_i(\lax), U'(\lax)\chi_i(\lax)\ran,
\end{align}
for $\lax\in I_i$, where $\chi_i(\lax)$ are well-chosen normalized eigenfunction of $H(\lax)$  corresponding to  the eigenvalue $\nu_i(\lax)$, namely that they are differentiable in $\lax$.
To this end, we observe that for each $\mu\in I_i$, we can find a local neighborhood $I_\mu\subset I_i$ of $\mu$ such that 
\begin{equation}
    \label{eq:no-intersection}
    \bigcup_{j\neq i}\{\nu_j(\lax) \,|\, \lax\in I_\mu \}\cap 
    \{ \nu_i(\lax) \,|\, \lax\in I_\mu \} = \emptyset,
\end{equation}
due to the isolated branch property, i.e., $\g_i>0$ in~\eqref{eq:BranchCond}.
Second, since $H(\lax)$ is self-adjoint for $\lax\in I_0$ and
 analytic (say, in the resolvent sense) in $\lax\in I_0^c$,  the Riesz projection, corresponding to  the eigenvalue~$\nu_i(\lax)$: 
\begin{equation} 
    \label{RieszProj}
    P_i(\lax) 
    := 
    \frac{1}{2 \pi i} \oint_{\Gamma_{i}(\mu)}  (H(\lax)-z)^{-1} dz,
\end{equation}
where $\Gamma_{i}(\mu)$ is a %n analytic 
closed curve in the resolvent set of $H(\lax)$ surrounding the eigenvalue branch $\{\nu_i(\lax): \lax\in I_\mu\}$, is also self-adjoint and  analytic in $\lax\in I_\mu^c$ and therefore in $\lax\in I_0^c$  (see \cite{RSIV,HS}),
condition~\eqref{eq:no-intersection} guarantees that we can choose such a closed curve which contains no other points of $\sigma(H(\mu))$ on $I_\mu$, 
and that, combining all neighborhoods of $\mu$ for $\mu \in I_i$, there holds that $P_i(\lax)$ is analytic in $I_i$. From~\cite[Theorem XII.12]{RSIV}, there exists an analytic family of unitary operators $V_i(\lax)$ such that $P_i(\lax) = V_i(\lax) P_i(\lax_0) [V_i(\lax)]^{-1}$, $\lax_0$ being possibly replaced by some arbitrary $\mu \in I_i$ if $\lax_0$ does not belong to $I_i$.
We then define 
\[
    \chi_i(\lax)
    =
    V_i(\lax) \psi_{0 i},
\]
where $\psi_{0 i}$ is an eigenvector of $H_0$ corresponding to the eigenvalue $\lax_0$.
Since  $V_i(\lax)$ is analytic, $\chi_i(\lax)$ is also analytic in $I_i$, so in particular differentiable, and one can easily check that $\chi_i(\lax)$ is of norm 1 and that $P_i(\lax) \chi_i(\lax) = \chi_i(\lax)$, which guarantees that $ \chi_i(\lax)$ is a normalized  eigenfunction of $H(\lax)$.
Now, we use that 
\begin{align*}
    \lan \chi_i'(\lax), H(\lax)\chi_i(\lax)\ran+\lan \chi_i(\lax), H(\lax)\chi_i'(\lax)\ran
    &=
    {\nu}_i(\lax)(\lan \chi_i'(\lax), \chi_i(\lax)\ran+\lan \chi_i(\lax), \chi_i'(\lax)\ran)\\
    &=
    \lan \chi_i(\lax), \chi_i(\lax)\ran'=0
\end{align*} 
to obtain ${\nu}_i'(\lax)=\lan \chi_i(\lax), H'(\lax)\chi_i(\lax)\ran$, which gives \eqref{nu-deriv}.
The differentiability of $\chi_i(\lax)$ and the analyticity of $H(\lax)$ then implies the differentiability of $\nu_i$ in each neighborhood of $\lax$.

In order to prove (ii), note that $U'(\lax)\le 0$, as follows by the explicit formula 
\begin{align}
	\label{Ulam-deriv'}
	U'(\lam):=-P W\oP (H^\perp- \lax)^{-2} \oP WP\le 0.
\end{align}
 Hence, ${\nu}_i'(\lax)<0$ by~\eqref{nu-deriv}. 
 The monotonicity of $\nu_i(\lax)$ also implies the well-posedness of the equations $\nu_i(\lax) =\lax$ under the condition that $\lax\in I_i$, thus statement (iii).

 We now aim to prove (iv). 
Starting from~\eqref{nu-deriv}, we estimate ${\nu}_i'(\lax)$ with
\begin{align}
	\label{nu'-est1-1_v1}
	|\nu_i'(\lax)| \le \|(H_0+ \aaa)^{1/2} P_0\|^2\|U'(\lax)\|_{H_0, \aaa}. 
\end{align}
The first factor on the right hand side is exactly known as
\begin{equation}
    \label{estim-H0a_v1}
     \|(H_0+ \aaa)^{1/2} P_0\|^2 = (\lax_0 +\alpha).
\end{equation}
To investigate the second factor on the r.h.s. of \eqref{nu'-est1-1_v1}, we use the analyticity $U(\lax)$ and the estimate \eqref{U-bnd}. Indeed, by the Cauchy integral formula, we have \[\|U'(\lax)\|_{H_0, \aaa}\le \frac{1}{2\pi R}\sup_{\substack{\mu \in \C, \\ |\mu-\lax|=R}} \|U(\mu)\|_{H_0, \aaa},\] where $R$ is such that  $\{\mu\in \C:|\mu-\lax|\le R\}\subset I_0^c$.
  Taking  $R=\frac14 \g_0$ gives, under the conditions of Lemma \ref{lem:FSM-conds}, the estimate 
\begin{align}
	\label{U'-bnd}
	&\|U'(\lax)\|_{H_0, \aaa}\le
    \frac{8}{\pi} \frac{\lac}{\g_0^2} \|P_0 W \oP\|_{H_0, \aaa}^2.
\end{align}
Combining equations \eqref{nu'-est1-1_v1}, \eqref{estim-H0a_v1} and \eqref{U'-bnd} shows (iv).
\end{proof}

\begin{cor}
\label{cor:ContractionH}
Let $\alpha\in\R$ be such that  $H_0+\aaa>0$ and let $I_i\subset I_0$ be such that the branch $\nu_i(\lax)$ is isolated on $I_i$.
  Under the condition that
  \[
    \frac{8}{\pi} \frac{(\lax_0 +\alpha)\lac}{\g_0^2} \|P_0 W \oP\|_{H_0, \aaa}^2 < 1,
  \]
  and that the unique solution $\lax$ of $\nu_i(\lax)=\lax$ satisfies $\lax\in\I'$, the fixed-point iteration $\lax^{(k+1)} = \nu_i(\lax^{(k)})$ converges to $\lax$ for initial values in $\I'$.
\end{cor}

Now, we proceed directly to the proof of Theorem \ref{thm:FS-pert-var}. 
\begin{proof}[Proof of Theorem \ref{thm:FS-pert-var}]
For the estimate on the eigenvalues, we first remark that applying Proposition~\ref{prop:eigenvalue-estimate} to the $m$ eigenvalues corresponding to $\lax_0$ for $H_0$ provides the first inequality in~\eqref{eq:eigenvalue_est}. 
The second one follows immediately from the condition $ \| W  \|_{\!H_0,\aaa} \le \frac{1}{2} \frac{\g_0}{\lax_0+\aaa}$ and thus $\mu_i\in I_0$.

The fact that the operator $H$ has exactly $m$ eigenvalues (counting the multiplicities) in $I_0$ follows from Corollary \ref{cor:nuFP} and Proposition \ref{prop:H-evs}(iii) and the fact that $H(\lax)$ is a $m\times m$ symmetric matrix.

For the estimates on the  eigenfunctions, recall from Theorem~\ref{thm:isospF} that $Q_0 (\mu_i)\psi_{0i}=\psi_{i}$, where $\mu_i = \nu_i(H)$ and the operator $Q_0 (\lax)$ is given by 
\begin{align} \label{Q}  
    Q_0 (\lax)  
    &:= 
    \one - R^\perp(\lam) \oP W \Pz,
\end{align}
with $R^\perp(\lam)$ defined in~\eqref{resolv-id2}.
This yields
 \begin{align}
	\psi_{ 0i } - \psi_{i}=\psi_{ 0i } - Q_0(\mu_{i})  \psi_{ 0i}=R^\perp(\mu_{i}) \oP W \Pz \psi_{ 0i }. 
 \end{align}
 Then, for $\gamma\in \{0,1/2\}$,
 \begin{align*}
     \| \Hza^\gamma(\psi_{ 0i } - \psi_{i}) \|
      \le \; & \| \Hza^\gamma R^\perp(\mu_{i}) \oP W \Pz\| \\
      \le \; & \| \Hza^\gamma \Hza^{-1/2} \oP \|  \| \Hza^{1/2} R^\perp(\mu_{i}) \Hza^{1/2} \| 
      \| \oP W \Pz\|_{H_0,\alpha} 
      \| \Pz \Hza^{1/2}\|.
 \end{align*}
In the previous expression, we can use~\eqref{ResPerp-bnd2} to estimate $\| \Hza^{1/2} R^\perp(\mu_{i}) \Hza^{1/2} \|$. Then, we note that
 \[
    \| \oP W \Pz\|_{H_0,\alpha} 
    \le \|W \|_{H_0,\alpha},
 \]
 as well as 
 \[
    \| \Pz \Hza^{1/2}\| = (\lax_0 + \alpha)^{1/2}.
 \]
 Finally, in the case $\gamma = 0$
 \[
    \| \Hza^{-1/2} \oP \| \le (\lax_0 + \alpha)^{-1/2},
 \]
 and for $\gamma = 1/2$,
 \[
   \| \Hza^\gamma \Hza^{-1/2} \oP \| \le 1.
 \]
 Combining the four bounds leads to~\eqref{eq:est-eigenvectors1} and~\eqref{eq:est-eigenvectors2}.
 \end{proof}

 \begin{remark}
 Note that by Theorem~\ref{thm:isospF}, any solution $\mu_i$ to the equation $\nu_i(\mu_i)=\mu_i$, for $i=1,\ldots,m$ and where the $\mu_i$ are in ascending order, is an eigenvalue of $H$.
Under the condition 
\begin{equation}
    \label{eq:WcondI}    
    \|W\|_{H_0, \aaa} \le \frac12 \frac{\g_0}{\lax_0 +  \aaa},
\end{equation}
Theorem~\ref{thm:FS-pert-var} guarantees that the eigenvalues \old{$\lax_i$}\new{$\mu_i$} satisfy $|\mu_i-\lax_0|\le \frac{\g_0}{2}$ and thus  $\mu_i \in \I$. 

On the contrary, the eigenvalues $\mu_i$ are the only $m$ eigenvalues of $H$ belonging to $\I$ if a similar condition as~\eqref{eq:WcondI}, but for the next larger eigenvalue of $H_0$ than $\lax_0$ holds whereas such a condition is automatically satisfied by \eqref{eq:WcondI} for the preceding eigenvalue of $H_0$.

\end{remark}
 
\section{Preliminary results} \label{sec:prelim-res} 

We now derive a few preliminary results that will be useful for proving Theorem~\ref{thm:main}.
For the following proofs, we define the following quantities:
 $h_\lax:=-\Delta-\lax$, $V_M^\perp = \PMp V \PMp$, $\VMN = \PNM V \PNM$.

\begin{lem}
\label{lem:Vperp-bnd}
For $\lax\in \C$ with $\re\lax < \frac12 \rho_M$ and $\rho_M \ge 1$, the following bounds hold
\begin{align}\label{q-bnd}
	\|h_\lax^{-1/2}\VMN h_\lax^{-1/2}\| \le \|h_\lax^{-1/2}V_M^\perp h_\lax^{-1/2}\|  \le 4  \laM^{-r}\Enorm{V}. 
\end{align}
\end{lem} 

\begin{proof} 
First, we note that 
\[
    	\|h_\lax^{-1/2}\VMN h_\lax^{-1/2}\| = \|P_N h_\lax^{-1/2} V_M^\perp h_\lax^{-1/2} P_N\| \le \| h_\lax^{-1/2}
    	 V_M^\perp 
    	h_\lax^{-1/2} \|.
\]
Then, we estimate for any $s\ge 0,$ using the assumption $\re\lax\le \frac12\laM $
\begin{align}\label{hal/hlam-est}
	&\|h_{-1}^s h_{\lam}^{-s} \PMp\|
	\le \frac{|\rho_M +1|^s }{|\rho_M - \lax|^s}
	\le 2^s \; ( 1 + \rho_M^{-1})^s \le 4^s,	
\end{align} 
This implies in particular that  
 $\|h_{\lam}^{-1/2}V_M^\perp h_{\lam}^{-1/2}\|
 \le 4 \, \EnormZ{V_M^\perp}$. The result follows noting that $\EnormZ{V_M^\perp} \le \laM^{-r}\Enorm{V}.$
\end{proof}

 \begin{lem}
\label{lem:UNK-bnd}
For $\lax\in\C$ with $\re\lax\le \min(\frac12\laM, \ka_M)$ and $\laM \ge 1$,
 the following bound holds  
\begin{align}
	\label{UNMK-bnd}
	\|\Usl\|_{r} \le 
	& \;
	4 \laM^{-r} \Enorm{V}^2   \sum_{k=0}^K \left[4  \laM^{-r}\Enorm{V} \right]^k.
\end{align}
Moreover, if $4  \laM^{-r}\Enorm{V} < 1$, 
\begin{align}
	\label{UNMK-bnd2}
	\|\Usl\|_{r} \le 
	& \;
	 \laM^{-r}  \frac{4\,\Enorm{V}^2}{1 - 4  \laM^{-r}\Enorm{V}},
\end{align}
and in particular
\begin{align}
 	    \label{UM-bnd}
      \|\UM(\lax)\|_{r} & \le \laM^{-r}  \frac{4\,\Enorm{V}^2}{1 - 4  \laM^{-r}\Enorm{V}}.
\end{align}
\end{lem}

\begin{proof} 
 	By definition 
 	\eqref{UNMKlam}, we can write $U_\sigma(\lax)$ which is well-defined for $\lax < \ka_M$ as
 \begin{align}
	\Us(\lax)
	= - \sum_{k=0}^K 
	\PM V \PNM h_{\lax}^{-1/2} \Big[ - h_{\lax}^{-1/2} \VMN h_{\lax}^{-1/2} \Big]^k h_{\lax}^{-1/2}  \PNM V \PM.
\end{align}
Using estimate \eqref{q-bnd}, there holds
\[
    \| \Us(\lax) \|_r
	\le   \sum_{k=0}^K \left[4  \laM^{-r}\Enorm{V} \right]^k
	\| h_{-1}^{-1/2 +r/2} \PM V \PNM h_{\lax}^{-1/2} \|    \| h_{\lax}^{-1/2}  \PNM V \PM h_{-1}^{-1/2 +r/2}\| .
\]
and by \eqref{hal/hlam-est}, $\|h_{\lam}^{-1/2}\PM^N V \PM h_{-1}^{-1/2+r/2}\| \le  2 \laM^{-r/2}\|V\|_{r}$, so that we obtain~\eqref{UNMK-bnd}. The bound~\eqref{UNMK-bnd2} is easily obtained from~\eqref{UNMK-bnd} 
and taking $K,N=\infty$ in \eqref{UNMK-bnd2}, we arrive at \eqref{UM-bnd}. 
\end{proof}

\begin{lem}
\label{lem:UN-bnd}
For $\lax < \frac12 \ka_M$, $\laM \ge 1$ and if $4  \laM^{-r}\Enorm{V} < 1$, the following bounds hold
 \begin{align}
      \label{UM'-est}
      \|\UM'(\lax)\|_{r } & \le \frac{1}{\pi (\ka_M - 2\lax)} \laM^{-r}  \frac{4\,\Enorm{V}^2}{1 - 4  \laM^{-r}\Enorm{V}}  .
 \end{align}
\end{lem}
\begin{proof} 
Since from Proposition \ref{prop:UN-prop}(iv), $\UM(\lax)$ is complex analytic in $\lax$ for $\re\lax <  \ka_M$, by the Cauchy integral formula, we have 
\[
    \|U_M'(\lax)\|_{r}\le \frac{1}{2\pi R_M(\lax)}\sup_{\substack{\mu \in \C, \\ |\mu-\lax|=R_M(\lax)}} \|U_M(\mu)\|_{r},
\] 
with $R_M(\lax) = \frac12\ka_M - \lax >0 $. Using~\eqref{UM-bnd} and noting that $\frac12 \ka_M\le \frac12\laM$, we obtain~\eqref{UM'-est}.
\end{proof}

\begin{lem}
\label{lem:UN-UNK-bnd}
For $\lax < \min(\ka_M,\frac12\laM)$, $\laN \ge \laM >1$  and if $4  \laM^{-r}\Enorm{V} < 1$ and $4  \laN^{-r}\Enorm{V} + \frac{16  \laM^{-2r}\Enorm{V}^2}{1 - 4 \rho_M^{-r} \Enorm{V}}<1$,
 the following bound holds
 \begin{align}
    \label{UM-Us-bnd}
&	\|\UM(\lax)-\Us(\lax)\|_{r}\notag\\ & \le  \;
          \frac{4 \laN^{-r} \Enorm{V}^2}{1 - 4  \laN^{-r}\Enorm{V} - \frac{16  \laM^{-2r}\Enorm{V}^2}{1 - 4 \rho_M^{-r} \Enorm{V}}} \left[ 1 +  \frac{4 \laM^{-r} \Enorm{V}}{1 - 4 \rho_M^{-r} \Enorm{V}} \right]^2
         + 
	\frac{4 \laM^{-r} \,\Enorm{V}^2}{1 - 4 \laM^{-r} \Enorm{V}} \left[4  \laM^{-r}\Enorm{V} \right]^{K+1}.
\end{align}
\end{lem}

\begin{proof}
We first write $\UM(\lax)-\Us(\lax)$ as 
\begin{align}
	\label{UN-Us-expr}
	&\UM(\lax)-\Us(\lax)
	=
	(\UM(\lax)-\UMNl) + (\UMNl-\Us(\lax)),
\end{align}
where, using the notation $\HMN = \PNM\Hcal\PNM$
\begin{align*}
    \UMNl
	:=&
	-\PM V\PNM (\HMN - \lax)^{-1} \PNM V\PM
	\\
	=&
	-\PM V\PNM (-\Delta + \VMN - \lax)^{-1} \PNM V\PM.
\end{align*}

Since $\UMl = -\PM V\PMp (\HMp- \lax)^{-1} \PMp V\PM$, with  $\HMp = \PMp \Hcal \PMp$,
the first term $\UM(\lax)-\UMNl$ is estimated as follows.

Denoting by $\calH_N^\perp = \PNp \calH \PNp$, and the Schur complement
\[
A = \PNp \left( (\calH_N^\perp - \lax \PNp ) - \PNp V \PNM (\calH_M^\perp - \lax)^{-1} \PNM V \PNp \right)^{-1} \PNp, 
\]
there holds, using a block matrix inversion
\begin{align}
    \PMp (\HMp- \lax)^{-1} \PMp  =& 
    \PNM (\HMN- \lax)^{-1} \PNM + \PNM (\HMN- \lax)^{-1} \PNM V A V \PNM (\HMN- \lax)^{-1} \PNM \nonumber \\
    & - \PNM (\HMN- \lax)^{-1} \PNM V A - A V \PNM (\HMN- \lax)^{-1} \PNM + A. 
    \label{eq:matrix_inversion}
\end{align}
Therefore, $\UM(\lax)-\UMNl$ can be decomposed into four terms as
\begin{align*}
    \UM(\lax)-\UMNl = & -\PM V \PNM (\HMN- \lax)^{-1} \PNM V A V \PNM (\HMN- \lax)^{-1} \PNM  V\PM, \\
    & + \PM V \PNM (\HMN- \lax)^{-1} \PNM V A  V\PM, \\
     & + \PM V A V \PNM (\HMN- \lax)^{-1} \PNM  V\PM, \\
       & -\PM V  A  V\PM.
\end{align*}
Then, the $r$-norm can be estimated as 
\begin{align*}
    \|\UM(\lax)-\UMNl\|_r  \le \; & 
    \Enorm{V}^2 \Big[  \| h_{-1}^{1/2-r/2} \PNM (\HMN- \lax)^{-1} \PNM V A V \PNM (\HMN- \lax)^{-1} \PNM 
    h_{-1}^{1/2-r/2} \| \\
    &
    + \| h_{-1}^{1/2-r/2} \PNM (\HMN- \lax)^{-1} \PNM V A 
    h_{-1}^{1/2-r/2} \| \\
    & + \| h_{-1}^{1/2-r/2} A V \PNM (\HMN- \lax)^{-1} \PNM 
    h_{-1}^{1/2-r/2} \| 
    \\
    &
    + 
    \| h_{-1}^{1/2-r/2} A
    h_{-1}^{1/2-r/2} \| 
    \Big].
\end{align*}
Introducing appropriate $ h_{-1}^{1/2-r/2}$ and  $h_{-1}^{-1/2+r/2}$ terms, we obtain 
\begin{align*}
    \|\UM(\lax)-\UMNl\|_r  \le \; & 
    \Enorm{V}^2 \Big[  \| h_{-1}^{1/2-r/2} \PNM (\HMN- \lax)^{-1} \PNM h_{-1}^{1/2-r/2} \|^2  \Enorm{V}^2 \| h_{-1}^{1/2-r/2} A h_{-1}^{1/2-r/2} \| \\
    &
    + 2 \, \| h_{-1}^{1/2-r/2} \PNM (\HMN- \lax)^{-1} \PNM h_{-1}^{1/2-r/2} \|  \Enorm{V} \| h_{-1}^{1/2-r/2} A 
    h_{-1}^{1/2-r/2} \| \\
    &
    + 
    \| h_{-1}^{1/2-r/2} A
    h_{-1}^{1/2-r/2} \| 
    \Big] \\
     \le \; & 
     \| h_{-1}^{1/2-r/2} A
    h_{-1}^{1/2-r/2} \|  \Enorm{V}^2 \\
    & \times \Big[ 1 +  \| h_{-1}^{1/2-r/2} \PNM (\HMN- \lax)^{-1} \PNM h_{-1}^{1/2-r/2} \|  \Enorm{V} \Big]^2.
\end{align*}
We are therefore left with the estimation of $ \| h_{-1}^{1/2-r/2} \PNM (\HMN- \lax)^{-1} \PNM h_{-1}^{1/2-r/2} \|$ and $\| h_{-1}^{1/2-r/2} A
    h_{-1}^{1/2-r/2} \|$.
    First, noting from~\eqref{q-bnd} that
\begin{align}
    \label{eq:HMperpbound_new}
    \| h_{\lambda}^{1/2} \PNM (\HMp- \lax)^{-1} \PNM 
    h_{\lambda}^{1/2} \| 
    & \le \| (I + h_{\lax}^{-1/2} V_M^\perp h_{\lax}^{-1/2} )^{-1}  \| 
     \le \frac{1}{1 - 4 \rho_M^{-r} \Enorm{V}},
\end{align}
and using that
\begin{align}
    \label{eq:AuxLem7}
     \| h_{-1}^{1/2-r/2} h_\lambda^{{-1/2}} P_M^N \| \le 2 \laM^{-r/2},
\end{align}
we obtain
\begin{equation}
    \| h_{-1}^{1/2-r/2} \PNM (\HMN- \lax)^{-1} \PNM h_{-1}^{1/2-r/2} \| \le
     \frac{4 \laM^{-r}}{1 - 4 \rho_M^{-r} \Enorm{V}}.
\end{equation}
Second, 
\begin{align*}
    \| h_{-1}^{1/2-r/2} A
    h_{-1}^{1/2-r/2} \| = \| h_{-1}^{1/2-r/2} \PNp \left[ (\calH_N^\perp - \lax \PNp ) - \PNp V \PNM (\calH_M^\perp - \lax)^{-1} \PNM V \PNp \right]^{-1} \PNp
    h_{-1}^{1/2-r/2} \|.
\end{align*}
Noting that 
\begin{align*}
    \| h_{-1}^{1/2-r/2} h_\lambda^{{-1/2}} P_N^\perp \| \le 2 \laN^{-r/2},
\end{align*}
there holds
\[
 \| h_{-1}^{1/2-r/2} A
    h_{-1}^{1/2-r/2} \| \le 4 \laN^{-r} \| h_\lambda^{{1/2}} \PNp \left[ (\calH_N^\perp - \lax \PNp ) - \PNp V \PNM (\calH_M^\perp - \lax)^{-1} \PNM V \PNp \right]^{-1} \PNp
    h_\lambda^{{1/2}} \|.
\]
Factorizing $h_\lambda$, we deduce
\[
 \| h_{-1}^{1/2-r/2} A
    h_{-1}^{1/2-r/2} \| \le 4 \laN^{-r} \big\| \big[ I + h_\lambda^{-1/2} V_N^\perp h_\lambda^{-1/2} -  h_\lambda^{-1/2} \PNp V \PNM (\calH_M^\perp - \lax)^{-1} \PNM V \PNp  h_\lambda^{-1/2} \big]^{-1} \big\|.
\]
From~\eqref{q-bnd} with $N$ in place of $M$, we obtain
\begin{equation}
    \label{eq:temp_3003}
        \| h_\lambda^{-1/2} V_N^\perp h_\lambda^{-1/2}\| \le  4  \laN^{-r}\Enorm{V}.
\end{equation}
Moreover, 
\[
    \| h_\lambda^{-1/2} \PNp V \PNM (\calH_M^\perp - \lax)^{-1} \PNM V \PNp  h_\lambda^{-1/2} \| 
    \le 
    \| h_\lambda^{-1/2} \PNp V \PNM  h_\lambda^{-1/2} \|^2
    \| h_\lambda^{1/2} \PNM (\calH_M^\perp - \lax)^{-1} \PNM h_\lambda^{1/2} \|,
\]
which, from~\eqref{q-bnd} and~\eqref{eq:HMperpbound_new}, leads to
\[
    \| h_\lambda^{-1/2} \PNp V \PNM (\calH_M^\perp - \lax)^{-1} \PNM V \PNp  h_\lambda^{-1/2} \| \le  \frac{16  \laM^{-2r}\Enorm{V}^2}{1 - 4 \rho_M^{-r} \Enorm{V}}.
\]
Combining this last line with~\eqref{eq:temp_3003}, we obtain the bound
\[
    \| h_{-1}^{1/2-r/2} A
    h_{-1}^{1/2-r/2} \| \le 4 \laN^{-r} \frac{1}{1 - 4  \laN^{-r}\Enorm{V} - \frac{16  \laM^{-2r}\Enorm{V}^2}{1 - 4 \rho_M^{-r} \Enorm{V}}}.
\]
This leads to the following bound for the difference $\UM(\lax)-\UMNl$
\begin{equation}
    \label{eq:UM-UMN}
    \|\UM(\lax)-\UMNl\|_r  \le 
         \laN^{-r} \frac{4 \Enorm{V}^2}{1 - 4  \laN^{-r}\Enorm{V} - \frac{16  \laM^{-2r}\Enorm{V}^2}{1 - 4 \rho_M^{-r} \Enorm{V}}} \left[ 1 +  \frac{4 \laM^{-r} \Enorm{V}}{1 - 4 \rho_M^{-r} \Enorm{V}} \right]^2.
\end{equation}

For the second term on the right handside of~\eqref{UN-Us-expr}, 
we write $\UMN(\lax)-\Us(\lax)$ as 
\begin{align}
    \label{eq:AuxLem7-2}
	\UMN(\lax)-\Us(\lax)
	= - \sum_{k=K+1}^\infty
	\PM V \PNM h_{\lax}^{-1/2} \Big[ - h_{\lax}^{-1/2} \VMN h_{\lax}^{-1/2} \Big]^k h_{\lax}^{-1/2}  \PNM V \PM.
\end{align}
Using~\eqref{q-bnd} and~\eqref{hal/hlam-est}, we obtain
\begin{align*}
	\| \UMN(\lax)-\Us(\lax) \|_r
	\le & \Enorm{V}^2 \; \|h_{-1}^{1/2-r/2} h_{\lambda}^{-1/2} P_M^\perp\|^2 \sum_{k=K+1}^\infty
	\| h_{\lax}^{-1/2} \VMN h_{\lax}^{-1/2} \|^k \\
	\le & 4 \, \Enorm{V}^2 \; \laM^{-r}  \sum_{k=K+1}^\infty
	 \left[4  \laM^{-r}\Enorm{V} \right]^k,
\end{align*}
from which we deduce that
\begin{equation}
    \label{eq:UMN-Us}
    \| \UMN(\lax)-\Us(\lax) \|_r
	\le \frac{4\, \Enorm{V}^2}{1 - 4  \laM^{-r}\Enorm{V}} \; \laM^{-r}  
	 \left[4  \laM^{-r}\Enorm{V} \right]^{K+1}.
\end{equation}
Combining~\eqref{eq:UM-UMN} and~\eqref{eq:UMN-Us}, we obtain~\eqref{UM-Us-bnd}.
\end{proof}

\begin{lem}
   \label{lem:UNdist}
   For $\lax < \mu < \min(\ka_M, \frac12\laM)$, $\laM \ge 1$ and if $4  \laM^{-r}\Enorm{V} < 1$,
   there holds
    \begin{align}
        \label{eq:UNdist}
        \|\UM(\mu) - \UM(\lax)\|_r
        &
        \le |\mu-\lax| \frac{\laM^{-r}}{\pi (\ka_M - 2\lax)}   \frac{4\,\Enorm{V}^2}{1 - 4  \laM^{-r}\Enorm{V}}
   \end{align}
\end{lem}
\begin{proof}[Proof of Lemma~\ref{lem:UNdist}]
Without loss of generality, we assume that $\lax<\mu$.
Writing 
\[
    \UM(\mu) - \UM(\lax)
    =
    \int_{\lax}^\mu\UM'(s) ds,
\]
yields
\[
    \|\UM(\mu) - \UM(\lax)\|_r
    \le |\mu-\lax| \max_{s\in [\lax,\mu]} \| \UM'(s) \|_r.
\]
We conclude by applying~\eqref{UM'-est} of Lemma~\ref{lem:UN-bnd}.
\end{proof}

\section{Proof of the main results}
\label{sec:main-proof}
The goal of this section is to provide the proof for Theorem~\ref{thm:main}.

We first prove the following technical lemmas which will be useful later.
\begin{lem}
\label{lem:HW}
For $H,W$ such that $H=-\Delta+W$, there holds for $\aaa\ge (2\,\|W\|_r)^{1/r}$ such that $H+\aaa>0$
\[
    \| A \|_{H,\aaa} 
    \le 
    2 \, \| A \|_{-\Delta,\aaa}
    \le 
    2 \alpha^{-r} \| A \|_{r}.
\]
\end{lem}
\begin{proof}
First, denoting again $h_{-\aaa}:=-\Delta+\aaa$, 
note that
\[
    \| A \|_{H,\aaa}
    \le
    \| A \|_{-\Delta,\aaa} \|h_{-\aaa}^{1/2}(H+\aaa)^{-1/2}\|^2.
\]
Let $v$ be arbitrary and define $u=(H+\aaa)^{-1/2}v$. Note that
\begin{align*}
    \|h_{-\aaa}^{1/2}u\|^2
    &= \langle u, h_{-\aaa} u \rangle
    = \langle u, (H+\aaa-W)u \rangle
    = \|v\|^2 -  \langle u,W u \rangle.
\end{align*}
Further, there holds
\[
 \langle u,W u \rangle \le \|W\|_{-\Delta,\aaa} \|h_{-\aaa}^{1/2}u\|^2,
\]
and using the inequality $\|W\|_{-\Delta,\aaa} \le \alpha^{-r} \|W\|_{r}$ we obtain that $\|W\|_{-\Delta,\aaa}\le \frac12$ for all $\aaa \ge (2\,\|W\|_r)^{1/r}$, and that $\|h_{-\aaa}^{1/2}(H+\aaa)^{-1/2}\|^2\le 2$ yielding the first inequality. 

The second inequality results from using, once again, the inequality $\|A\|_{-\Delta,\aaa} \le \alpha^{-r} \|A\|_{r}$.
\end{proof}

Before starting the proof of Theorem~\ref{thm:main}, we analyze the relation of the spectra of $\calH$ and $\HMl$. 
\begin{lem}
\label{lem:HN-gap-bs} 
Let $\lax_i<\lax_j$ be the $i$-th resp. $j$-the eigenvalue of $\calH$ and let $\nu_{Mi}(\lax)$ denote the $i$-th eigenvalue of the operator $\HMl$. 
Then, we have the following $M$-independent lower bound 
\begin{align} 
    \label{gap-bnd-bs}
    \lax_j-\lax_i\le \min(\nu_{Mj}(\lax_i)-\nu_{Mi}(\lax_i),\nu_{Mj}(\lax_j)-\nu_{Mi}(\lax_j)).
\end{align}
\end{lem}
\begin{proof} 
Since  $\lax_i=\nu_{Mi}(\lax_i)$ and $\lax_j=\nu_{Mj}(\lax_j)$, we have $\lax_j-\lax_i=\nu_{Mj}(\lax_j)-\nu_{Mi}(\lax_i)$. 
By definition \eqref{HNlam-def} and Proposition \ref{prop:UN-prop}, the family  $\HMl$ is  monotonically decreasing with $\lam$. Hence $\nu_{Mi}(\lax_j)\le  \nuN^i(\lax_i)$ and $\nu_{Mj}(\lax_j)\le \nu_{Mj}(\lax_i)$, so that  $\lax_j-\lax_i=\nu_{Mj}(\lax_j)-\nu_{Mi}(\lax_i)$ implies \eqref{gap-bnd-bs}.
\end{proof}
\begin{cor}
\label{cor:HN-gap} 
Let $\lax_i$ be an isolated eigenvalue of $H$. Then, the gap $\g_{Mi}$ of $\lax_i$ to the rest of the spectrum of $\HM(\lax_i)$ is bounded below by the gap $\g_0$ of $\lax_i$ to the rest of the spectrum of $H$.
\end{cor}

\begin{lem}
\label{lem:PertEstim}
Let be $\aaa>0$ such that 
\begin{equation}
    \label{eq:Acond}
    \HM(\las)+\aaa>0,
    \qquad\mbox{and}\qquad
    \aaa>(2\,\|\UM(\las)-\Us(\laxsi)\|_r)^{1/r},
\end{equation}
$M$ large enough such that $\las,  \laxsi < \min(\ka_M, \frac12\laM)$, $\laM \ge 1$ and $4\laM^{-r}\Enorm{V}\le \frac13$, and $N \ge M$.

Then, there holds
\begin{equation}
    \label{eq:PertEstim}
    \|\UM(\las)-\Us(\laxsi) \|_{\HM(\las),\aaa}
    \le 
    |\laxsi-\las| 
    \frac{8\, \alpha^{-r}\laM^{-r}\Enorm{V}^2}{\pi (\ka_M - 2 \min(\las,\laxsi))} 
    +
    36 \, \frac{\|V\|_r^2}{\alpha^{r}}\,\varepsilon(\sigma,r,V)
\end{equation}
with, using the notation $\sigma=(N,M,K)$,
\[
    \varepsilon(\sigma,r,V)
    :=
    % \alpha^{-r}
	  \laN^{-r} + \laM^{-r}  
	 \left[ 4  \laM^{-r}\Enorm{V} \right]^{K+1}.
\]
\end{lem}

\begin{proof}
Applying Lemma~\ref{lem:HW} to the present case, we obtain
\begin{align*}
    \|\UM(\las)-\Us(\laxsi)\|_{\HM(\las),\aaa}
    &\le 
    2 \alpha^{-r} \|\UM(\las)-\Us(\laxsi)\|_{r}.
\end{align*}
Employing the triangle inequality, there holds
\begin{equation}
    \label{eq:ProofThm1TriangleIneq}
    \|\UM(\las)-\Us(\laxsi)\|_{r}
    \le 
    \|\UM(\las)-\UM(\laxsi)\|_{r}
    +
    \|\UM(\laxsi)-\Us(\laxsi)\|_{r}.
\end{equation}
For the first term of the right handside, Lemma~\ref{lem:UNdist} yields
\begin{align*}
    \|\UM(\laxsi) - \UM(\las)\|_r
    &\le 
    |\laxsi-\las| 
    \frac{8\laM^{-r}\Enorm{V}^2}{\pi (\ka_M - 2\min(\las,\laxsi))}   .
\end{align*}
Finally, we use Lemma~\ref{lem:UN-UNK-bnd} in order to bound the second term of~\eqref{eq:ProofThm1TriangleIneq}, we conclude that
\[
    \|\UM(\laxsi)-\Us(\laxsi)\|_{r}
    \le
 \,\Enorm{V}^2  \left[ 
	{18} \laN^{-r} + 6 \laM^{-r}  
	 \left[4  \laM^{-r}\Enorm{V} \right]^{K+1}
	\right].
\]
Combining all estimates concludes the proof.
\end{proof}

Finally, we are now ready to prove Theorem~\ref{thm:main}.

\begin{proof}[Proof of Theorem~\ref{thm:main}]

Note that $\las=\nuMi(\las)$, $\laxsi = \nusi(\laxsi)$ where $\nuMi(\lax)$ and $\nusi(\lax)$ represents the $i$-th eigenvalue of $\HM(\lax)$ and $\Hsi(\lam)$ respectively.

We choose $\aaa>0$ in order to satisfy the condition~\eqref{eq:Acond} of Lemma~\ref{lem:PertEstim}.
By the expression~\eqref{eq:PertEstim} of Lemma~\ref{lem:PertEstim} we can choose $\widetilde M_{0}\in \mathbb N $ such that
\begin{equation*}
    \|\UM(\las)-\Us(\laxsi) \|_{\HM(\las),\aaa}
    \le 
    \frac14 \frac{\g_0}{\lac},
\end{equation*}
for any $M\ge \widetilde M_{0}$, and where $\g_0$ denotes the gap of $\las$ to the rest of the spectrum of $\Hcal$,  which is a lower bound of $\g_{Mi}$ by Corollary~\ref{cor:HN-gap}, and $\lac = \las+\g_0+\aaa$. 
In consequence, we can apply Theorem~\ref{thm:FS-pert-var} with 
\begin{align}
    \label{eq:ApplThm3}
    H_0&=\HM(\las), \qquad
    W=\Us(\laxsi)-\UM(\las),  \qquad
    \lax_0 = \las,
\end{align}
and thus $H=\Hsi(\laxsi)$, yielding
\begin{align}
    \label{eq:EvalEstimPr1-1}
    |\las - \laxsi |
    &=
    |\nuMi(\las) - \nusi(\laxsi)|
    \le 
    (\las + \aaa) \|\UM(\las)-\Us(\laxsi) \|_{\HM(\las),\aaa}
    \\&\le 
    \label{eq:EvalEstimPr1-2}
    (\las + \aaa)
    \left[
        |\laxsi-\las| 
    \frac{16 \,\alpha^{-r}\laM^{-r}\Enorm{V}^2}{\pi (\ka_M - 2\min(\las,\laxsi))} 
    +
    { 54}\,\frac{\|V\|_r^2}{\alpha^{r}}\,\varepsilon(\sigma,r,V)
    \right]
\end{align}
It exists a $M_0\in\mathbb N$, with $M_0\ge \widetilde M_{0}$, such that
\begin{align}
    \label{as:boundby12}
      4\laM^{-r}\Enorm{V}\le \frac13
    \qquad\mbox{and}\qquad
    (\las + \aaa)\frac{16\, \alpha^{-r}\laM^{-r}\Enorm{V}^2}{\pi (\ka_M - 2\min(\las,\laxsi))} 
    \le \frac12,
\end{align}
for all $M\ge M_0$
and thus
\begin{equation}
    \label{eq:EVestimPr}
    |\las - \laxsi | 
    \ls (\las + \aaa)
    \frac{\|V\|_r^2}{\alpha^{r}} \, \varepsilon(\sigma,r,V).
\end{equation}
A similar development for the eigenfunctions based on the estimates~\eqref{eq:est-eigenvectors1}--\eqref{eq:est-eigenvectors2} of Theorem~\ref{thm:FS-pert-var} can be applied.
Indeed, we denote by $\varphi_{\sigma i}$ the $i$-th normalized eigenfunction of $\Hsi(\laxsi)$, thus, in terms of the notation used in Theorem~\ref{thm:FS-pert-var}, we have $\psi_i=\varphi_{\sigma i}$.
Then, the corresponding eigenfunction in the span of all eigenfunctions corresponding to $\las$ is given by $\varphi_{Mi} = {\sf P}_0 \varphi_{\sigma i}$, where ${\sf P}_0$ is the projector onto this span of all eigenfunctions of $\HM(\las)$ corresponding to $\las$.
Thus, $\varphi_{Mi}$ is an eigenfunction of $\HM(\las)$ associated to the eigenvalue $\las$.
Taking, again, Corollary~\ref{cor:HN-gap} into account yields
\begin{align}
   \| \varphi_{Mi} - \varphi_{\sigma i} \|
    & \le 4 \frac{\lac}{\g_0} \| \UM(\las)-\Us(\laxsi) \|_{\!\HM(\las),\aaa}, \label{eq:est-eigenvectors2-2}
\end{align}
Using the bounds of $\|\UM(\las)-\Us(\laxsi) \|_{\!\HM(\las),\aaa}$ expression~\eqref{eq:PertEstim} of Lemma~\ref{lem:PertEstim} and combining with~\eqref{eq:EVestimPr} yields the auxiliary result (see~\eqref{eq:est-eigenvectors1-0}): for $s\ge 0$ and for all $M\ge M_0$, 
\begin{align}
    \| (-\Delta+\id)^s (\varphi_{Mi} - \varphi_{\sigma i}) \|
    \ls 
    \rho_M^s   \| \varphi_{Mi} - \varphi_{\sigma i} \|
    & \ls \frac{\lac}{\g_0}\frac{\|V\|_r^2}{\alpha^{r}}\rho_M^s\varepsilon(\sigma,r,V) .
	\label{eq:est-eigenvectors2-3}
\end{align}

\medskip

Now we define the eigenfunction $\varphi_i := \QM(\las) \varphi_{Mi}$ where $\QM$ is defined by~\eqref{eq:QM}. 
Following Corollary~\ref{cor:isospFN}, $\varphi_i$ is an eigenfunction of $\Hcal$ associated to the eigenvalue $\las$.
Note that
\begin{align*}
    \varphi_i - Q_\sigma(\laxsi) \varphi_{\sigma i}
    = 
    \varphi_{Mi} - \varphi_{\sigma i}
    - \big[ S_M(\las) \varphi_{Mi} - S_\sigma(\laxsi)\varphi_{\sigma i} \big],
\end{align*}
with $S_M(\lax):=\PMp(\HMp-\lax)^{-1}\PMp V \PM$, $S_\sigma(\lax):=R_\sigma(\lax) \PNM V \PM$. 
Applying the triangle inequality several times yields
\begin{equation}
    \label{eq:Qansatz}
    \|\varphi_i - Q_\sigma(\laxsi) \varphi_{\sigma i}\|
    \le 
    (1+I_1)\|\varphi_{Mi} - \varphi_{\sigma i}\|
    + I_2 + I_3,
\end{equation}
with
\begin{align*}
    I_1 &=
    \| S_M(\las) \| ,
    \qquad
    I_2 =
    \| S_M(\las) - S_M(\laxsi) \|,
    \qquad 
    I_3 = 
    \| S_M(\laxsi) - S_\sigma(\laxsi) \|.
\end{align*}
For $I_1$, proceeding as in the proof of Lemma~\ref{lem:UNK-bnd}, 
writing 
\[
    	S_M(\lax)
	= \sum_{k=0}^\infty 
	 \PMp h_{\lax}^{-1/2} \Big[ - h_{\lax}^{-1/2} V_M^\perp h_{\lax}^{-1/2} \Big]^k h_{\lax}^{-1/2}  \PMp V \PM.
\]
and under the assumptions $\lax\le \min(\frac12\laM, \ka_M)$, $\rho_M\ge 2$, $4\laM^{-r}\Enorm{V}\le \frac12$, using the estimates of Lemma~\ref{lem:Vperp-bnd} $\| \PMp h_{\lax}^{-1/2} \|\le \sqrt{2} \rho_M^{-1/2}$, $\| \PM h_{-1}^s \|\le (\rho_M+1)^s \le (3\rho_M/2)^s$ for $s\ge 0$, one obtains
\begin{align}
    \label{eq:SMestim1}
    \| S_M(\lax) \| 
    & \le 
    2 \, \| \PMp h_{\lax}^{-1/2} \| \, 
    \| \PMp h_{\lax}^{-1/2} \PMp V \PM h_{-1}^{-1/2+r/2} \|
    \, \| \PM h_{-1}^{1/2-r/2} \|
    \\
    \label{eq:SMestim2}
    &\le 
    2 \sqrt{2} (3/2)^{1/2-r/2} \rho_M^{-1/2}  \rho_M^{-r/2} \|V\|_r \rho_M^{1/2-r/2}
    \le 12 \rho_M^{-r} \|V\|_r.
\end{align}
For $I_{2}$, we proceed as in Lemma~\ref{lem:UNdist} based on the results of Lemma~\ref{lem:UN-bnd} but substituting $U_M(\lax)$ by $S_M(\lax)$ in order to obtain
\begin{equation}
    \label{eq:I2}
     I_{2}
     \le 
     \frac{|\las - \laxsi|}{\pi (\kappa_M - 2\min(\las,\laxsi))} 12 \rho_M^{-r} \|V\|_r,
\end{equation}
based upon the estimate $\|S_M(\lax)\| \le 12 \rho_M^{-r} \|V\|_r$ from~\eqref{eq:SMestim1}--\eqref{eq:SMestim2}.

Finally for $I_3$, we proceed as in the proof of Lemma~\ref{lem:UN-UNK-bnd}. 
Indeed, we apply, once again, the triangle inequality to obtain $I_3\le I_{3,1} + I_{3,2}$ with
\begin{align*}
    I_{3,1} = \| S_M(\laxsi) - S_{MN}(\laxsi) \|
    \qquad
    I_{3,2} = \| S_{MN}(\laxsi) - S_{\sigma}(\laxsi) \|.
\end{align*}
with $S_{MN}(\lax):=\PNM ({\HMN}-\lax)^{-1} \PNM V \PM$.
  To estimate $I_{3,1}$, we first note that
\begin{align*}
    \nonumber
    \|S_M(\lax) - S_{MN}(\lax)\|
    \le
    &
    \| h_{-1}^{-1/2+r/2} \PNM \| 
    \| h_{-1}^{1/2-r/2} \left[ \PMp (\HMp- \lax)^{-1} \PMp  - 
    \PNM (\HMN- \lax)^{-1} \PNM \right] h_{-1}^{1/2-r/2} \| \\
    & \| h_{-1}^{-1/2+r/2}  V \PM \|  
\end{align*}
Using that $\| h_{-1}^{-1/2+r/2} \PNM \| \le \rho_M^{-1/2+r/2}$ and $\| h_{-1}^{-1/2+r/2}  V \PM \| \le \Enorm{V} \rho_M^{1/2-r/2}$, we obtain 
\[
    \|S_M(\lax) - S_{MN}(\lax)\|
    \le 
    \Enorm{V}
    \| h_{-1}^{1/2-r/2} \left[ \PMp (\HMp- \lax)^{-1} \PMp  - 
    \PNM (\HMN- \lax)^{-1} \PNM \right] h_{-1}^{1/2-r/2} \|.
\]
Adapting the proof of~\eqref{eq:UM-UMN}, and noting that $4  \laN^{-r}\Enorm{V} - \frac{16  \laM^{-2r}\Enorm{V}^2}{1 - 4 \rho_M^{-r} \Enorm{V}} <1$, and $4 \rho_M^{-r} \Enorm{V}<1$, we obtain
\begin{equation}
    \label{eq:I31}
       I_{3,1} \ls 
    \laN^{-r}  \Enorm{V}.
\end{equation}

Finally, for $I_{3,2}$, we proceed as in the derivation of~\eqref{eq:UMN-Us}, starting from a similar ansatz than~\eqref{eq:AuxLem7-2}
\[
	S_{MN}(\lax)-S_\sigma(\lax)
	= \sum_{k=K+1}^\infty
     \PNM h_{\lax}^{-1/2} \Big[ - h_{\lax}^{-1/2} \VMN h_{\lax}^{-1/2} \Big]^k h_{\lax}^{-1/2}  \PNM V \PM,
\]
 yielding, together with $\|\PNM h_{\lax}^{-1/2}\| \ls \rho_M^{-1/2}$ and 
 $\|h_{\lax}^{-1/2}  \PNM V \PM\| \ls \rho_M^{1/2-r} \Enorm{V}$,
\begin{align}
    \label{eq:I32}
    I_{3,2} = \| S_{MN}(\laxsi) - S_{\sigma}(\laxsi) \|
    \ls 
    \rho_M^{-r} \| V \|_r \left[4  \laM^{-r}\Enorm{V} \right]^{K+1}.
\end{align}
Starting from~\eqref{eq:Qansatz} and combining \eqref{eq:SMestim1}--\eqref{eq:SMestim2}, \eqref{eq:I2}, \eqref{eq:I31}, \eqref{eq:I32} with the estimates~\eqref{eq:est-eigenvectors2-3} (with $s=0$), \eqref{eq:EVestimPr} and the bound~\eqref{as:boundby12} concludes the proof. 
\end{proof}

\begin{remark}
    \label{rem:ConvFP}
     Having now understood how to apply Theorem~\ref{thm:FS-pert-var} in this context, i.e. using~\eqref{eq:ApplThm3}, and seen the abstract theory of Section~\ref{sec:pert-est} we can now make a statement about the convergence of the fixed-point iteration schemes~\eqref{EVPNMKk} and \eqref{EVPNMKk-2}: 
     Following Corollary~\ref{cor:ContractionH}, we note that if $\nusi(\lax)$ parametrizes an isolated branch on some interval $I_i$ containing $\laxsi$, $M\ge M_0$ and, using~\eqref{eq:EvalEstimPr1-1}--\eqref{eq:EvalEstimPr1-2} and~\eqref{eq:EVestimPr}, if $\varepsilon(\sigma,r,V)$ is small enough,
      then, the fixed-point iterations converge for any starting point in $I_i$. 
     Unfortunately, it is difficult to assess whether the isolated branch property holds in practical application and this remains an abstract result.
\end{remark}

\section{Numerical results}
\label{sec:NumRes}
In this section, we test the theoretical estimates developed in this article, {\blue first in a one-dimensional case, and second, in a three-dimensional case for a Coulomb potential}. 

\subsection{One-dimensional case}
We first consider a one-dimensional test case with $\Omega=(0,1)$ and a potential $V_t$ given by its Fourier coefficients:
\[
	\widehat V_0 = -10,\qquad
	\widehat V_n = -\frac{5}{|n|^t},
\]
so that $V_t\in \Hs{t-\tfrac{1}{2}-\varepsilon}$for any $\varepsilon>0$. 
We then consider the two different values $t=1$ or $t=0$, see Figure~\ref{fig:potential} for a graphical illustration for the case $t=1$, so that including the embedding expressed under the constraints~\eqref{eq:SobEmb} yields that $\Enorm{V}<\infty$ for $r =1$ if $t=1$ and for all $r < 1/2$ if $t=0$.

Note that, to the best of our knowledge, the classical convergence analysis does not cover such low regularities of the potentials. While the standard analysis can probably be extended to potentials in $L^2$ and thus covering the case $t=1$ (although we are not aware of any published analysis in this case), it certainly does not hold without further developments for $t=0$.

In Figure~\ref{fig:convK}, we illustrate the convergence of the discrete solutions with respect to $K$ for different values of $M$ and $t=1$ and $t=0$, and for fixed $N=500$.
The error in the eigenvalue and eigenvector are defined as
\[
    {\sf err}_{\sf val} = |\las-\laxsi|
    \qquad 
    \mbox{and}
    \qquad
    {\sf err}_{\sf vec} = \|\varphi_{i} - Q_\sigma(\laxsi)\varphi_{\sigma i}\|,
\]
where $(\las,\varphi_{M i})$, $(\laxsi,\varphi_{\sigma i})$ are the $i$-th solution to \eqref{EVPN} and \eqref{EVPNMK} respectively, using the computational Strategy 1 defined by~\eqref{EVPNMKk} targeting the smallest eigenvalue. 
The ``exact'' solution is obtained by computing the variational approximation for $N_{\sf e}=1000$.

 We observe that, in agreement with the theory for small enough $K,M$, the convergence rate with respect to $K$ for different values of $M$ is the same for the eigenvalue and eigenvector error and that the convergence rate improves with increased values of $M$.
 In this example, we observe that the condition $M\ge M_0$ is not restrictive and convergence can be observed for all values of $M$.
 It is also noted, in particular for the higher values of $M$ (but still very moderate), that the number on the truncation order $K$ can be kept very low to achieve a good accuracy. 
 This, in turn, means that the number of computations on the fine grid (essentially matrix-vector products involving the fine grid for each matrix-vector product involving the coarse Hamiltonian $\Hsi(\lax)$) can be kept to a minimum.
 
A similar behaviour is reported in Figure~\ref{fig:convK_n3} for the third eigenvalue using the potential $V_{t=1}$ and again $N=500$.

The number of required SCF-iterations~\eqref{EVPNMKk} to converge to an increment in the eigenvalue smaller than $10^{-12}$, thus very tight, is stable and very moderate over all test cases as reported in Tables~\ref{tab:results} and~\ref{tab:results2} (the case of $t=0$ behaves similarly and is not reported here).

Finally, Figure~\ref{fig:convNres} illustrates the error in the eigenvalue and eigenvector with respect to~$N$ for different values of $K$ and $M$ for the first eigenvalue with $V_{t=1}$ and $V_{t=0}$. 
We observe two regimes. First, for small values of $N$, the error is limited by the small size of the fine grid $\XN$, i.e. by a small $N$, and decreases with increasing $N\in [25,500]$. 
Second, when $N$ is large enough, the error due to the moderate values of $M$, $K$ is dominating and the error stagnates.
As $M$ or $K$ increase, the transition of the two regimes moves to lower accuracy. This agrees well with the theoretical result presented in Theorem~\ref{thm:main}.
For the potential $V_{t=1}$, we observe that the convergence rate in $N$ (for $M,K$ large enough) is roughly 3 for the eigenvalue error and 2.5 for the eigenvector error, which are the rates predicted by the standard analysis as outlined in Corollary~\ref{rem:TriangleIneq}.
For the less regular case $V_{t=0}$, and thus $\Enorm{V}<\infty$ for any $r<1/2$, we observe a rate of roughly 1 in both cases, eigenvalue and eigenvector $L^2$-error, which is exactly as predicted by Theorem~\ref{thm:main}.

\begin{figure}[t!]
	\centering
    \includegraphics[trim = 0mm 0mm 0mm 0mm, clip,width=0.6\textwidth]{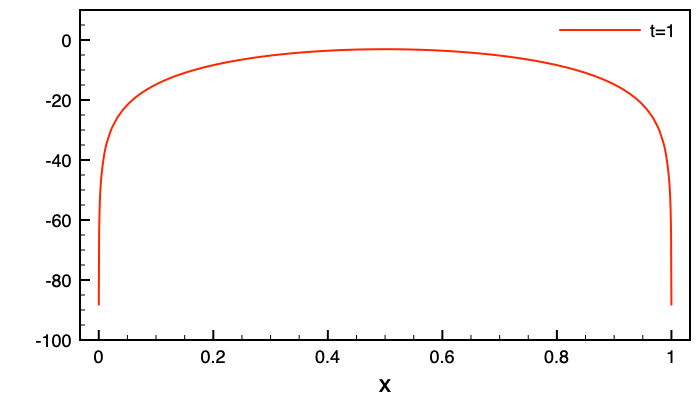}
	\caption{Illustration of the potential $V_{t=1}$.}
	\label{fig:potential}
\end{figure}

\begin{figure}[t!]
	\centering
    \includegraphics[trim = 0mm 0mm 0mm 0mm, clip,width=0.45\textwidth]{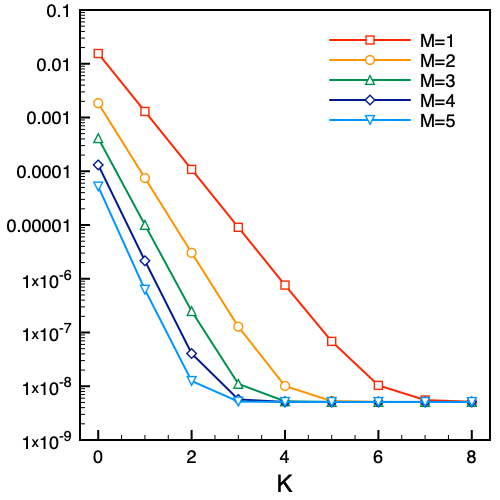}
     \includegraphics[trim = 0mm 0mm 0mm 0mm, clip,width=0.45\textwidth]{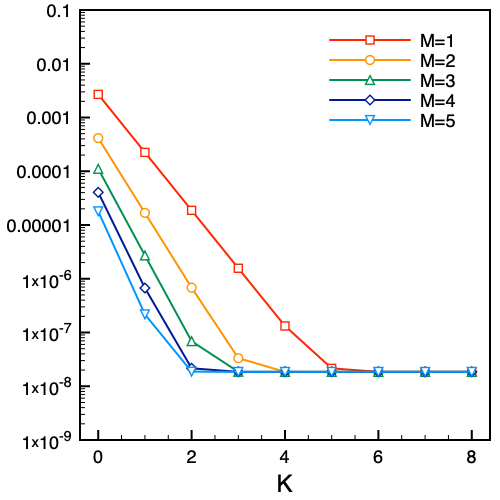}
    \\
    \includegraphics[trim = 0mm 0mm 0mm 0mm, clip,width=0.45\textwidth]{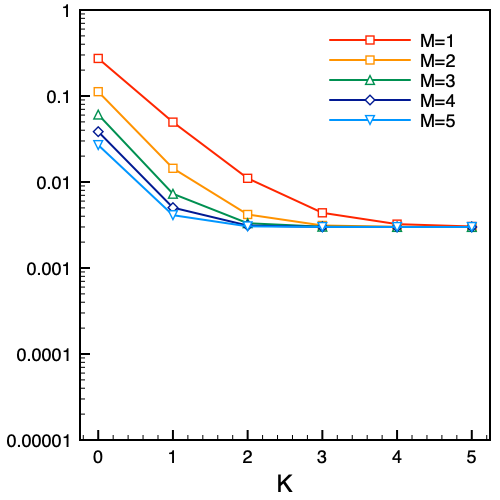}
    \includegraphics[trim = 0mm 0mm 0mm 0mm, clip,width=0.45\textwidth]{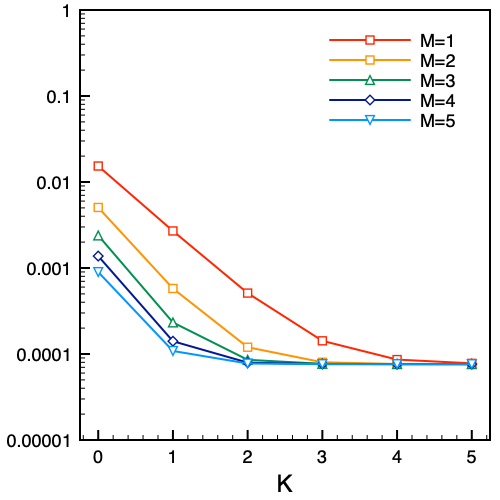}
	\caption{
	The convergence of the eigenvalue error ${\sf err}_{\sf val}$ (left) and the eigenvector error ${\sf err}_{\sf vec}$ (right)  corresponding to the first eigenvalue with respect to $K$ for different values of $M$ and for fixed $N=500$ \new{in the one-dimensional case.}
	The top row corresponds to a potential $V_{t=1}$ whereas the bottom row corresponds to~$V_{t=0}$.}
	\label{fig:convK}
\end{figure}

%%%%%%. negative potential. %%%%
\begin{table}[t]
	\footnotesize
	\setlength{\tabcolsep}{4pt}
	\centering
		\vspace{5pt}
      % \hspace*{-1cm}
      \begin{tabular}{cr}
		\hline
		\multicolumn{1}{c}{$M=1$}
		\\
		\multicolumn{1}{c}{$K$} &
%		\multicolumn{1}{c}{eff.} &
		\multicolumn{1}{c}{\# scf}
%		\Ttab\Btab
	 	\\
		\hline
1	&	6	\\
2	&	6	\\
3	&	6	\\
4	&	6	\\
\hline
	\end{tabular}
      \begin{tabular}{cr}
		\hline
		\multicolumn{1}{c}{$M=2$}
		\\
		\multicolumn{1}{c}{$K$} &
%		\multicolumn{1}{c}{eff.} &
		\multicolumn{1}{c}{\# scf}
%		\Ttab\Btab
	 	\\
		\hline
1	&		5	\\
2	&		5	\\
3	&		5	\\
4	&		5	\\
\hline
	\end{tabular}
     \begin{tabular}{cr}
		\hline
		\multicolumn{1}{c}{$M=3$}
		\\
		\multicolumn{1}{c}{$K$} &
%		\multicolumn{1}{c}{eff.} &
		\multicolumn{1}{c}{\# scf}
%		\Ttab\Btab
	 	\\
		\hline
1	&		4	\\
2	&		4	\\
3	&		4	\\
4	&		4	\\
\hline
	\end{tabular}
\begin{tabular}{cr}
		\hline
		\multicolumn{1}{c}{$M=4$}
		\\
		\multicolumn{1}{c}{$K$} &
% 		\multicolumn{1}{c}{eff.} &
		\multicolumn{1}{c}{\# scf}
%		\Ttab\Btab
	 	\\
		\hline
1	&		4	\\
2	&		4	\\
3	&		4	\\
4	&		4	\\
\hline
\end{tabular}
\begin{tabular}{cr}
		\hline
		\multicolumn{1}{c}{$M=5$}
		\\
		\multicolumn{1}{c}{$K$} &
% 		\multicolumn{1}{c}{eff.} &
		\multicolumn{1}{c}{\# scf}
%		\Ttab\Btab
	 	\\
		\hline
1	&		4	\\
2	&		4	\\
3	&		4	\\
4	&		4	\\
\hline
\end{tabular}
\begin{tabular}{cr}
		\hline
		\multicolumn{1}{c}{$M=6$}
		\\
		\multicolumn{1}{c}{$K$} &
% 		\multicolumn{1}{c}{eff.} &
		\multicolumn{1}{c}{\# scf}
%		\Ttab\Btab
	 	\\
		\hline
1	&		4	\\
2	&		4	\\
3	&		4	\\
4	&		4	\\
\hline
\end{tabular}
	% 	\vspace{0.5cm}
% 	\\
% 	     \begin{tabular}{crr}
% 		\hline
% 		\multicolumn{1}{c}{$M=1$}
% 		\\
% 		\multicolumn{1}{c}{$K$} &
% 		\multicolumn{1}{c}{eff.} &
% 		\multicolumn{1}{c}{\# scf}
% %		\Ttab\Btab
% 	 	\\
% 		\hline
% 1	&	4.792	&	7	\\
% 2	&	4.806	&	7	\\
% 3	&	4.808	&	7	\\
% 4	&	4.808	&	7	\\
% \hline
% 	\end{tabular}
%       \begin{tabular}{crr}
% 		\hline
% 		\multicolumn{1}{c}{$M=2$}
% 		\\
% 		\multicolumn{1}{c}{$K$} &
% 		\multicolumn{1}{c}{eff.} &
% 		\multicolumn{1}{c}{\# scf}
% %		\Ttab\Btab
% 	 	\\
% 		\hline
% 1	&	8.562	&	6	\\
% 2	&	8.581	&	6	\\
% 3	&	8.582	&	6	\\
% 4	&	8.583	&	6	\\
% \hline
% 	\end{tabular}
%      \begin{tabular}{crr}
% 		\hline
% 		\multicolumn{1}{c}{$M=3$}
% 		\\
% 		\multicolumn{1}{c}{$K$} &
% 		\multicolumn{1}{c}{eff.} &
% 		\multicolumn{1}{c}{\# scf}
% %		\Ttab\Btab
% 	 	\\
% 		\hline
% 1	&	12.462	&	6	\\
% 2	&	12.479	&	6	\\
% 3	&	12.480	&	6	\\
% 4	&	12.481	&	6	\\
% \hline
% 	\end{tabular}
%     \begin{tabular}{crr}
% 		\hline
% 		\multicolumn{1}{c}{$M=4$}
% 		\\
% 		\multicolumn{1}{c}{$K$} &
% 		\multicolumn{1}{c}{eff.} &
% 		\multicolumn{1}{c}{\# scf}
% %		\Ttab\Btab
% 	 	\\
% 		\hline
% 1	&	16.419	&	5	\\
% 2	&	16.434	&	5	\\
% 3	&	16.435	&	5	\\
% 4	&	16.435	&	5	\\
% 		\hline
% 	\end{tabular}
	\vspace{10pt}
	\caption{
The number of SCF-like iterations of the solution strategy~\eqref{EVPNMKk} to reach an increment of $10^{-12}$ in the eigenvalue for the approximation of the first eigenvalue with $V_{t=1}$ using $N=500$ in the one-dimensional case.
	}
	\label{tab:results}
\end{table}

%%%%%%. negative potential. %%%%
\begin{table}[t]
	\footnotesize
	\setlength{\tabcolsep}{4pt}
	\centering
		\vspace{5pt}
      % \hspace*{-1cm}
      \begin{tabular}{cr}
		\hline
		\multicolumn{1}{c}{$M=1$}
		\\
		\multicolumn{1}{c}{$K$} &
% 		\multicolumn{1}{c}{eff.} &
		\multicolumn{1}{c}{\# scf}
%		\Ttab\Btab
	 	\\
		\hline
1	&		6	\\
2	&		6	\\
3	&		6	\\
4	&		6	\\
\hline
	\end{tabular}
      \begin{tabular}{cr}
		\hline
		\multicolumn{1}{c}{$M=2$}
		\\
		\multicolumn{1}{c}{$K$} &
% 		\multicolumn{1}{c}{eff.} &
		\multicolumn{1}{c}{\# scf}
%		\Ttab\Btab
	 	\\
		\hline
1	&		4	\\
2	&		4	\\
3	&		4	\\
4	&		4	\\
\hline
	\end{tabular}
     \begin{tabular}{cr}
		\hline
		\multicolumn{1}{c}{$M=3$}
		\\
		\multicolumn{1}{c}{$K$} &
% 		\multicolumn{1}{c}{eff.} &
		\multicolumn{1}{c}{\# scf}
%		\Ttab\Btab
	 	\\
		\hline
1	&		4	\\
2	&		4	\\
3	&		4	\\
4	&		4	\\
\hline
	\end{tabular}
    \begin{tabular}{cr}
		\hline
		\multicolumn{1}{c}{$M=4$}
		\\
		\multicolumn{1}{c}{$K$} &
% 		\multicolumn{1}{c}{eff.} &
		\multicolumn{1}{c}{\# scf}
%		\Ttab\Btab
	 	\\
		\hline
1	&		4	\\
2	&		4	\\
3	&		4	\\
4	&		4	\\
\hline
	\end{tabular}
\begin{tabular}{cr}
		\hline
		\multicolumn{1}{c}{$M=5$}
		\\
		\multicolumn{1}{c}{$K$} &
% 		\multicolumn{1}{c}{eff.} &
		\multicolumn{1}{c}{\# scf}
%		\Ttab\Btab
	 	\\
		\hline
1	&		3	\\
2	&		3	\\
3	&		3	\\
4	&		3	\\
\hline
\end{tabular}
\begin{tabular}{cr}
		\hline
		\multicolumn{1}{c}{$M=6$}
		\\
		\multicolumn{1}{c}{$K$} &
% 		\multicolumn{1}{c}{eff.} &
		\multicolumn{1}{c}{\# scf}
%		\Ttab\Btab
	 	\\
		\hline
1	&		3	\\
2	&		3	\\
3	&		3	\\
4	&		3	\\
\hline
\end{tabular}
	\vspace{10pt}
	\caption{	
	The number of SCF-like iterations of the solution strategy~\eqref{EVPNMKk} to reach an increment of $10^{-12}$ in the eigenvalue for the approximation of the third eigenvalue with $V_{t=1}$ using $N=500$ in the one-dimensional case.
}
	\label{tab:results2}
\end{table}

\begin{figure}[t!]
	\centering
    \includegraphics[trim = 0mm 0mm 0mm 0mm, clip,width=0.45\textwidth]{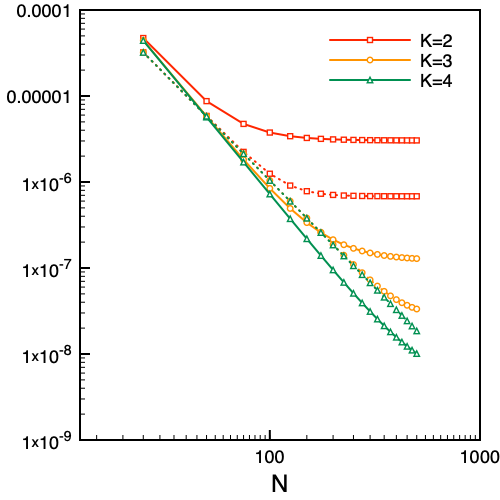}
    \includegraphics[trim = 0mm 0mm 0mm 0mm, clip,width=0.45\textwidth]{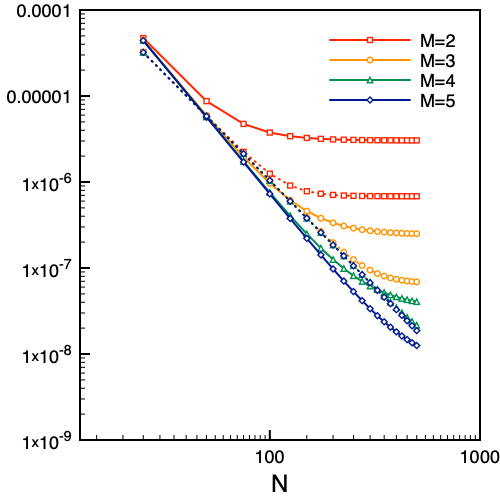}
    \includegraphics[trim = 0mm 0mm 0mm 0mm, clip,width=0.45\textwidth]{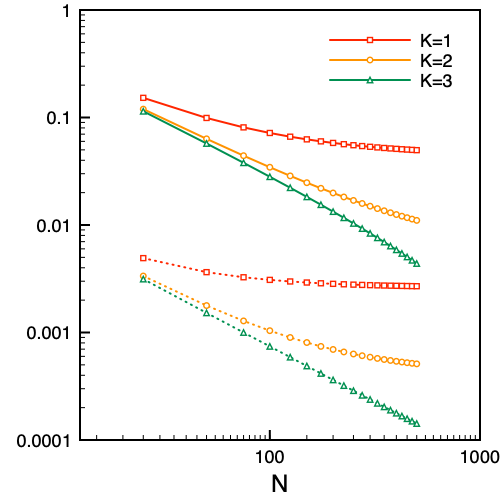}
    \includegraphics[trim = 0mm 0mm 0mm 0mm, clip,width=0.45\textwidth]{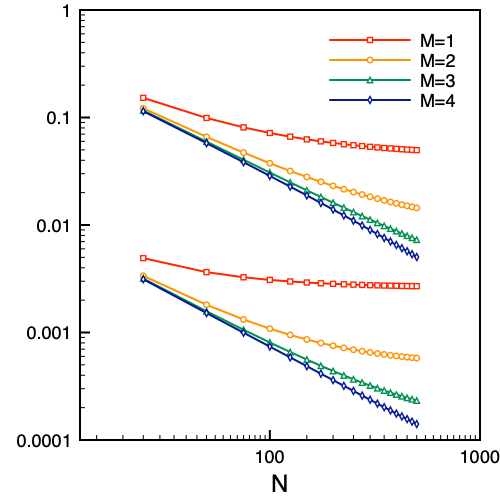}
	\caption{
 	The convergence of the eigenvalue error ${\sf err}_{\sf val}$ (straight lines) and the eigenvector error  ${\sf err}_{\sf vec}$ (dotted lines) corresponding to the first eigenvalue with respect to $N$ for different values of $K$ (left) with $M=2$ (top) and $M=1$ (bottom) and different values of $M$ (right) with $K=2$ (top) and $K=1$ (bottom) for the potential $V_{t=1}$ (top) and $V_{t=0}$ (bottom), \new{in the one-dimensional case.}
 	}
	\label{fig:convNres}
\end{figure}

\begin{figure}[t!]
	\centering
    \includegraphics[trim = 0mm 0mm 0mm 0mm, clip,width=0.45\textwidth]{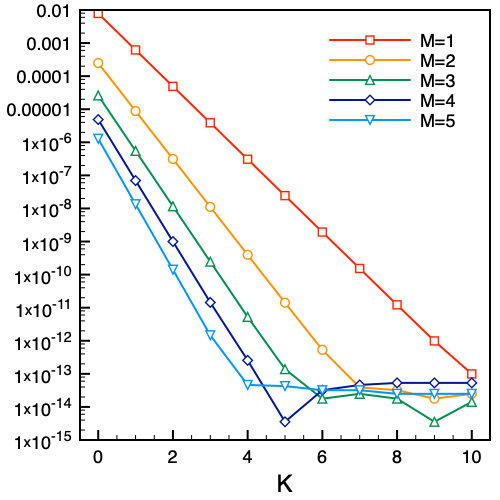}
    \includegraphics[trim = 0mm 0mm 0mm 0mm, clip,width=0.45\textwidth]{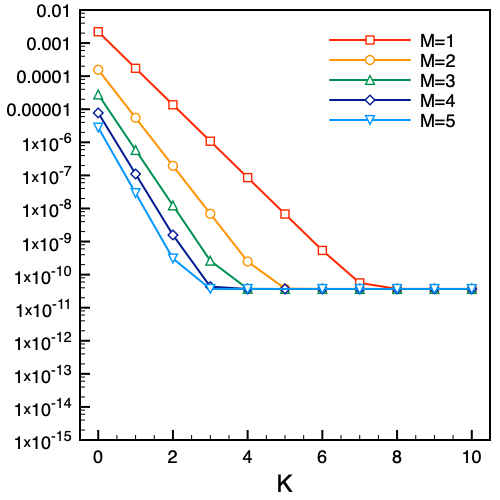}
	\caption{
		The convergence of the eigenvalue error ${\sf err}_{\sf val}$ (left) and the eigenvector error ${\sf err}_{\sf vec}$ (right)  corresponding to the third eigenvalue with respect to $K$ for different values of $M$ and for fixed $N=500$ and with $V_{t=1}$, \new{in the one-dimensional case.}
    }
	\label{fig:convK_n3}
\end{figure}

\subsection{Three-dimensional case}

\blue 

We then performed similar numerical tests for a three-dimensional problem, in order to test cases closer to real applications. Note that, as in the rest of the article, we consider a linear eigenvalue problem. We focus on the lowest eigenvalue of the Schr\"odinger hamiltonian. 
To do so, we used the code Density Functional Theory Toolkit (\texttt{DFTK.jl}) \cite{Herbst2021-yw}, which provides a spectral Fourier discretization for Density Functional Theory eigenvalue problems, and which can easily be adapted to treat linear eigenvalue problems. Since our method is valid for a large range of potentials, including low-regularity ones, we provide numerical simulations for the Coulomb potential, that is given in Fourier space as
\[
    	\widehat V_0 = 10,\qquad
	\widehat V_{\mathbf{k}} = -\frac{10}{|\mathbf{k}|^2}.
\]
The lattice is chosen as a cubic lattice with a lattice constant of 10 Bohr. Since the electronic structure theory uses a different convention to define the approximation space $\XN$, we now switch to a slightly different terminology and refer to the energy cutoff ${\rm EcN}=\frac12\rho_N$, corresponding to the largest eigenvalue/kinetic energy of the operator $-\frac12\Delta$ in the approximation space $\XN$, to fix the values of $N$.

In order to compute the reference (``exact'') solution an energy cutoff ${\rm EcN}_{\rm ref} = 225$ Ha is taken, which corresponds to choosing a reference grid of size $N_{\rm ref} \simeq 21.2$, with 161'235 Fourier coefficients.

We obtain similar results as in the one-dimensional case. Namely, we observe in Figure~\ref{fig:convK_3d} that for small values of $K$, the convergence rate with respect to the energy cutoff ${\rm EcM}$ is the same for the eigenvalue and the eigenvector and increases  when ${\rm EcM}$ increases, as predicted by the theory. For large values of $K$, we observe a saturation of the error, which is due to the error with respect to $N$ dominating, as stated in Theorem~\ref{thm:main}.

In Figure~\ref{fig:convNres_3d}, we show the analogous of Figure~\ref{fig:convNres} for the one-dimensional case, which is quite similar, as we also observe two regimes. For small values of ${\rm EcN},$ most of the error comes from the error in $N$, while for large values of ${\rm EcN},$ the error mainly comes from the error in $M$ and $K$. However, due to the slow convergence with respect to the grid size, we do not observe the convergence rates in ${\rm EcN}$ as precisely as in the one-dimensional case, this being mostly due to the huge size of the reference basis needed to achieve good convergence. 

Moreover, the iterative algorithm used to solve the problem converges very fast. In practice, we do not need to go beyond 7 SCF iterations in all tested cases to reach a $10^{-12}$ increment on the eigenvalue. 

\black

\begin{figure}[t!]
	\centering
    \includegraphics[trim = 0mm 0mm 0mm 0mm, clip,width=0.45\textwidth]{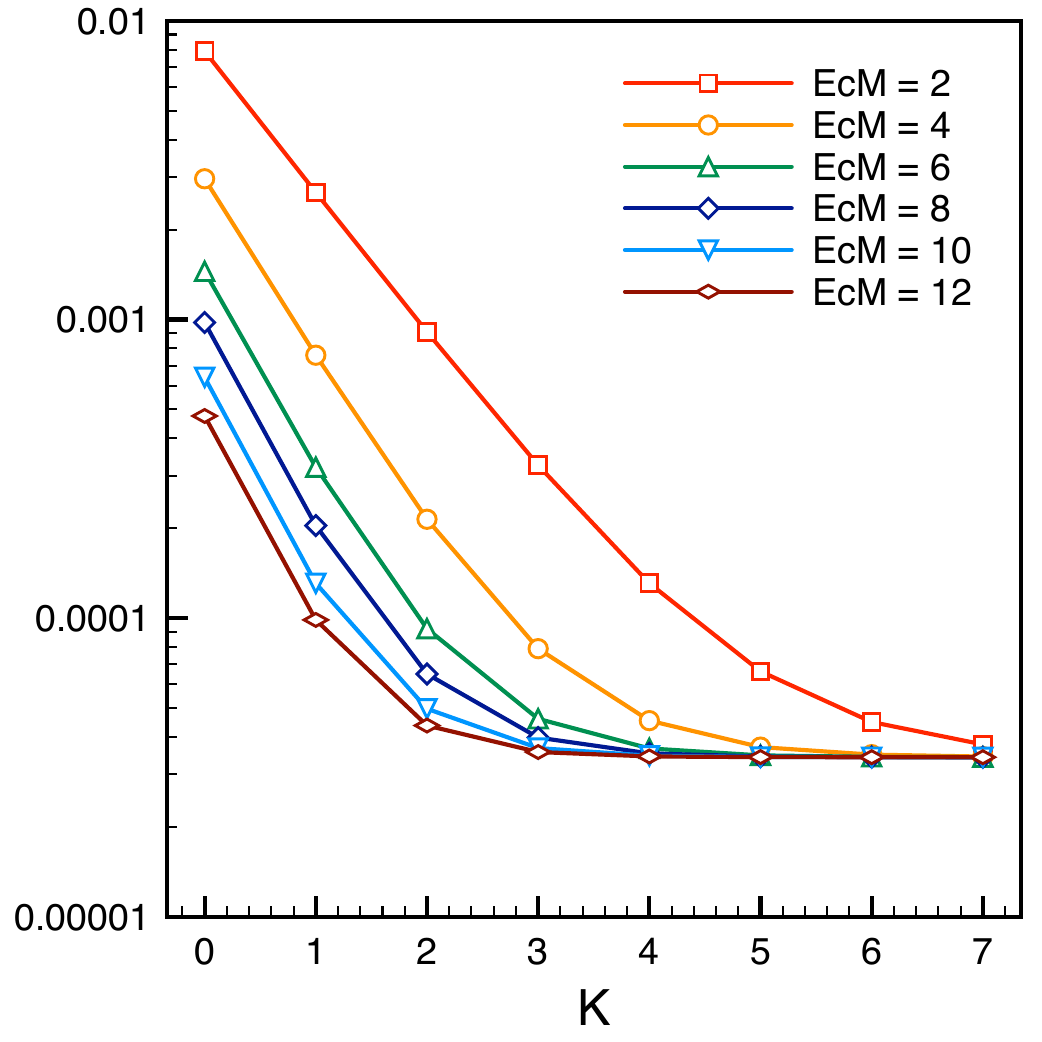}
     \includegraphics[trim = 0mm 0mm 0mm 0mm, clip,width=0.45\textwidth]{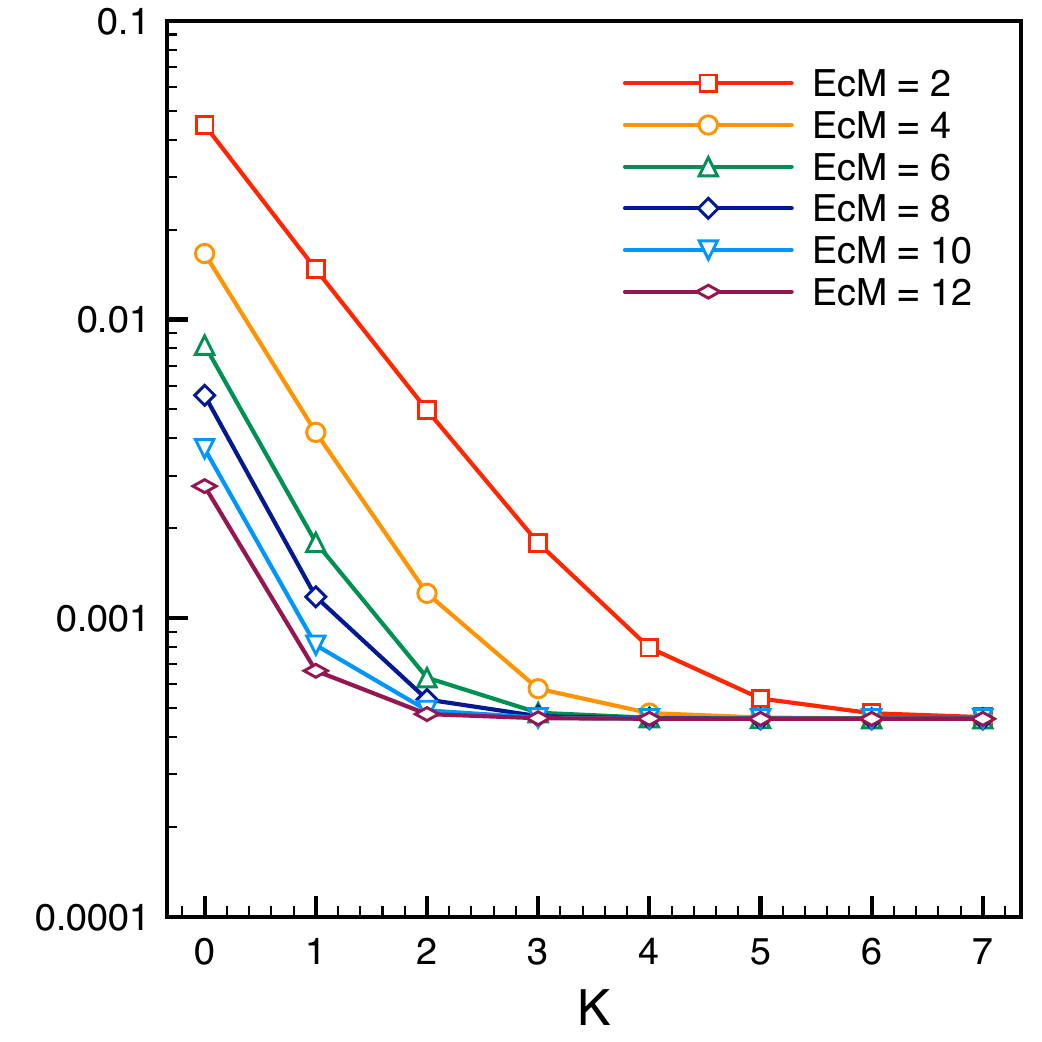}
    % \\
    % \includegraphics[trim = 0mm 0mm 0mm 0mm, clip,width=0.45\textwidth]{ConvK_s0_eval.png}
    % \includegraphics[trim = 0mm 0mm 0mm 0mm, clip,width=0.45\textwidth]{ConvK_s0_evec_q.png}
	\caption{
	{\blue
	The convergence of the eigenvalue error ${\sf err}_{\sf val}$ (left) and the eigenvector error ${\sf err}_{\sf vec}$ (right)  corresponding to the first eigenvalue with respect to $K$ for different values of energy cutoffs ${\rm EcM}$ and for fixed ${\rm EcN} = 155$, for a (three-dimensional) \blue Coulomb potential.
    }}
	\label{fig:convK_3d}
\end{figure}

\begin{figure}[t!]
	\centering
    \includegraphics[trim = 0mm 0mm 0mm 0mm, clip,width=0.45\textwidth]{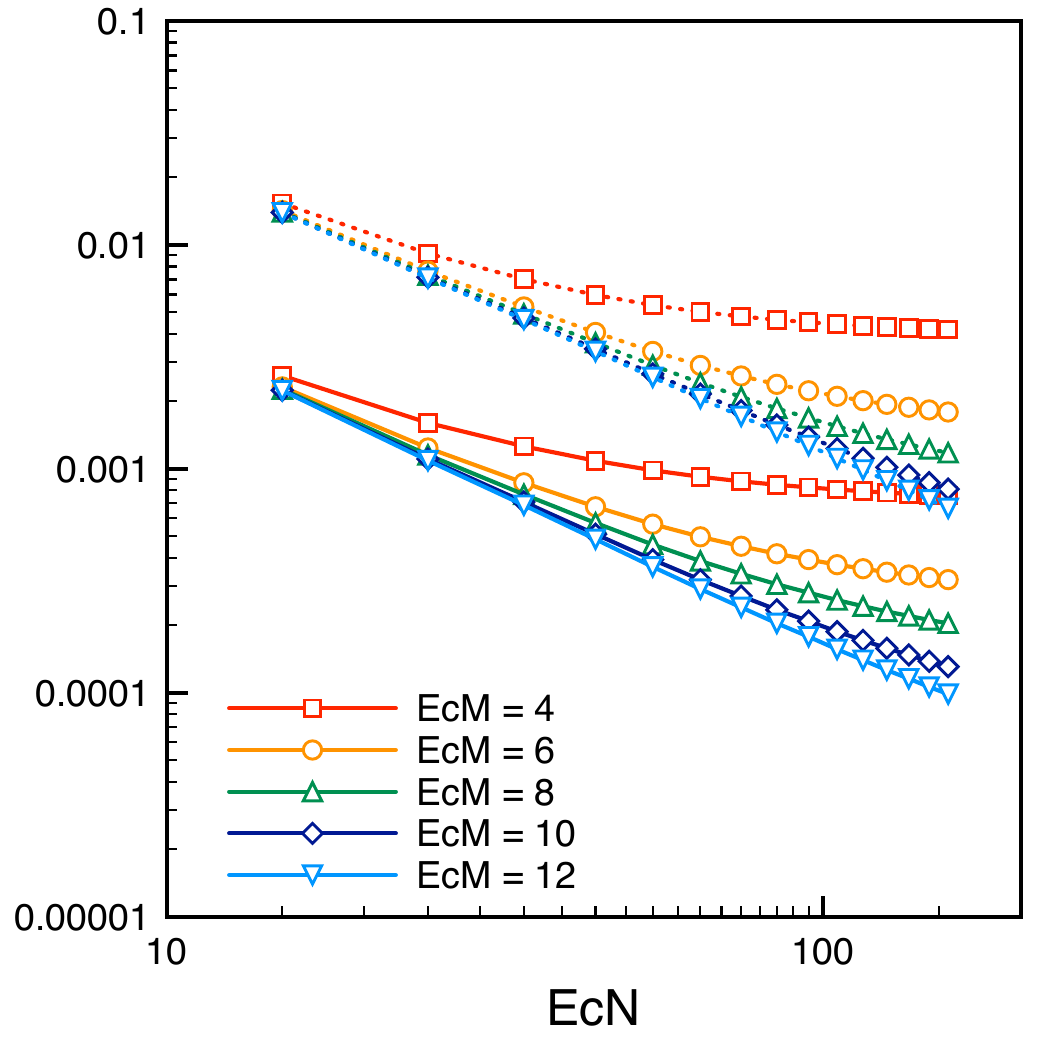}
    \includegraphics[trim = 0mm 0mm 0mm 0mm, clip,width=0.45\textwidth]{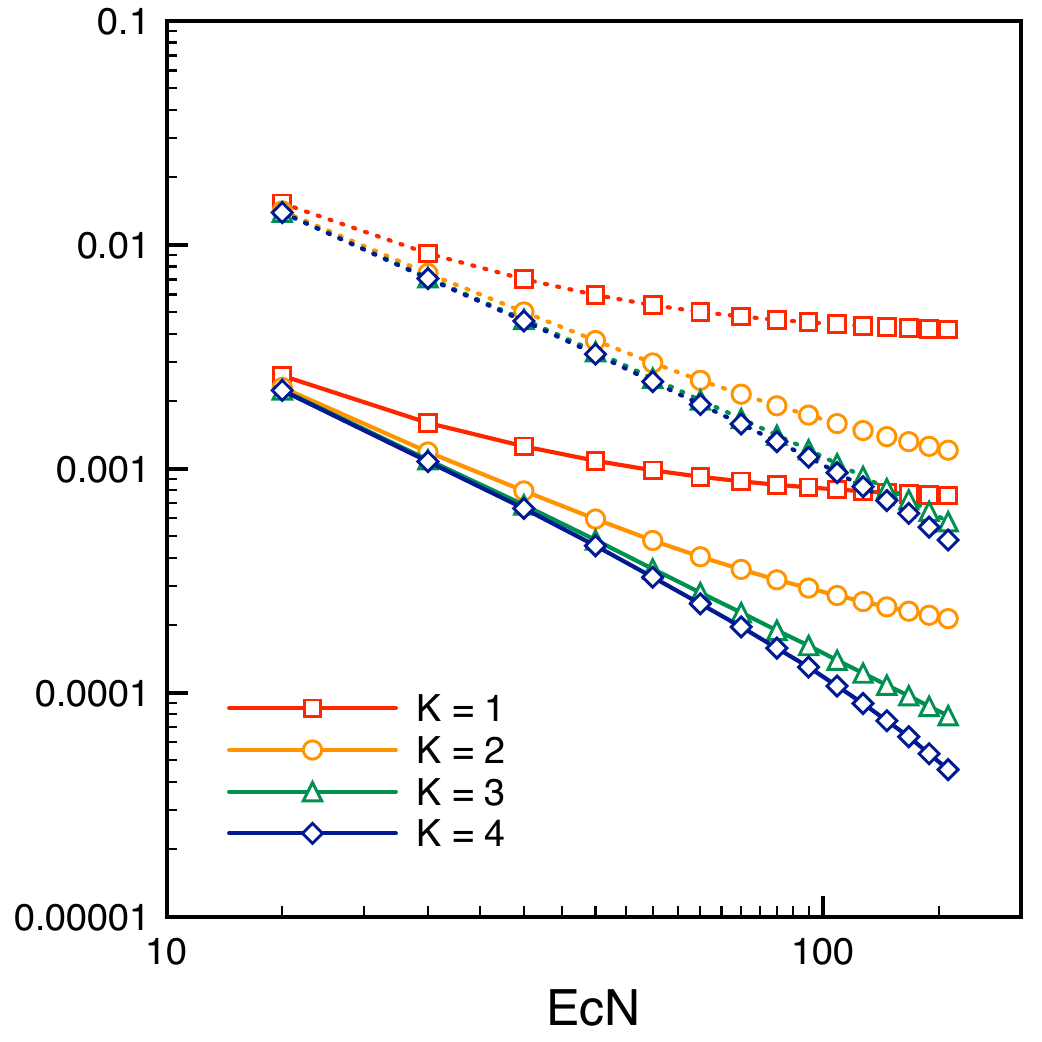}
	\caption{
	{\blue 
 	The convergence of the eigenvalue error ${\sf err}_{\sf val}$ (straight lines) and the eigenvector error  ${\sf err}_{\sf vec}$ (dotted lines) corresponding to the first eigenvalue with respect to the energy cutoff ${\rm EcN}$ for different values of ${\rm EcM}$  (left) with $K=1$ and different values of  $K$ \blue (right) with ${\rm EcM}=4$ for a (three-dimensional)  Coulomb potential.
 	}}
	\label{fig:convNres_3d}
\end{figure}

\section{Conclusion and perspectives}
\label{sec:sec6}

In this paper, we have proposed a new numerical method based on the Feshbach-Schur map in combination with \old{planewave}
\new{the spectral Fourier} discretizations for linear Schr\"odinger eigenvalue problems. 
The method does not rely on the variational principle but reformulates the infinite-dimensional problem as an equivalent problem, non-linear in the spectral parameter, on a finite dimensional grid whose unknowns are the exact eigenvalue and the best-approximation of the exact eigenfunctions on the given grid. 
Such a problem can then be approximated by evaluating the Feshbach-Schur map on a second finer grid. 
The substantial contribution of this paper is an analysis in order to provide error estimates of the proposed method  in all discretization parameters. 

For this, we developed in Section~\ref{sec:pert-est} a version of perturbation theory that relies on the notion of form-boundedness with increased regularity, as stated by Assumption~\ref{as:pot}.

Having established the method and its analysis, its full benefits shall be further analyzed in future. 
At the present stage, it is worth to mention that, for the considered \new{one- and three-dimensional (with a Coulomb potential)} problems, the contraction in the perturbation is rather small and the non-linear iteration converge rapidly. 
Also, in view of more sophisticated non-linear eigenvalue problems, the artificial extra non-linearity does not seem to be much of a burden.

The future developments include the extension of Section~\ref{sec:pert-est} to a more general family of operators, including non-symmetric perturbations of self-adjoint operators as well as extending the numerical method and its analysis to cluster of eigenvalues using a density-matrix based formulation.

\bibliographystyle{abbrv}
\bibliography{biblio}

\appendix

\section{Technical results and proofs} \label{sec:tech-res} 
We present proofs of some technical statements made in the previous sections and some additional  technical results. In what follows, we denote $V_N^\perp := \PNp V\PNp$.
\begin{lem}
\label{lem:H-low-bnd}
Under Assumption~\ref{as:pot}, (i) the norm $\|V\|_{0,\bb} :=\|(-\Delta+\bb)^{-1/2}V(-\Delta+\bb)^{-1/2}\|$ is well-defined for $\alpha >0$ and $\|V\|_{0,\bb} \rightarrow 0$ when $\alpha \rightarrow +\infty$ and (ii) $\Hcal$ is bounded below as 
\begin{align}
      \label{Hlowbnd}
      \Hcal&\ge[1-\|V\|_{0,\bb}](-\Delta+\alpha)- \bb,
\end{align}\end{lem}

\begin{proof}  
The first statement is obvious and the second one follows from the identity
\begin{align}
    \label{H-repr}
\Hcal=(-\Delta+\bb)^{1/2}[\one+(-\Delta+\bb)^{-1/2}V(-\Delta+\bb)^{-1/2}](-\Delta+\bb)^{1/2}- \bb
\end{align}
and the fact that the expression in square braces is bounded below by $1-\|V\|_{0,\bb}$. \end{proof}

\begin{proof}[Proof of Lemma \ref{lem:Hperp-low-bnd}]
Proceeding as in~\eqref{Hlowbnd},
we obtain for any $\alpha \in \R$ large enough so that $\|V_M^\perp\|_{0,\bb} <1$ 
\begin{align*}
      -\Delta + V_M^\perp \ge & [1-\|V_M^\perp\|_{0,\bb}](-\Delta+\alpha)- \bb. 
\end{align*}
Using that $-\Delta\ge (2\pi M/L)^2=\laM$ on $\Ran \PM^\perp$ and the estimate
\begin{align*}
       \|V_M^\perp\|_{0,\bb}
    & \le 
    \|(1-\Delta)^{1/2} (-\Delta+\alpha)^{1/2} \PMp\|^2
    \|(1-\Delta)^{-1/2} \PMp  V \PMp (1-\Delta)^{-1/2}\| \\
    & \le 
    \frac{\laM + 1}{\laM + \alpha}\laM^{-{r}} \, \Enorm{V},
\end{align*}
we deduce
\begin{align*}
    -\Delta + V_M^\perp &\ge \left( 1- \frac{\laM + 1}{\laM + \alpha}\laM^{-{r}} \, \Enorm{V} \right) (\laM + \alpha) - \alpha, \\
     &\ge \laM + \alpha - (\laM + 1) \laM^{-{r}} \, \Enorm{V} 
     - \alpha, \\
     &\ge \laM - (\laM + 1) \laM^{-{r}} \, \Enorm{V},
\end{align*}
from which we obtain the result.
\end{proof}

Now, we introduce the Sobolev spaces $\Hs{s}$ of periodic real  functions resp. distributions which can be characterized in a simple way using Fourier series: for $s\in\R$, we have
\begin{align}
	\Hs{s} := &\left\{ v  = \sum_{\bk \in \cR^\ast} \widehat v_\bk e_\bk \; \bigg| \; \forall \; \bk\in\cR^\ast:\quad \widehat v_{-\bk} = \overline{\widehat v_\bk}, \quad
	\|v\|_{\Hss{s}} < \infty \right\},
	\label{eq:Hss}
\end{align}
where the norm $\|v\|_{\Hss{s}}$ is given by the $\Hs{s}$ inner product is defined by
\[
	\forall u,v\in \Hs{s}, \quad    (u,v)_{\Hss{s}} := \sum_{\bk \in \cR^\ast} (1+|\bk|^2)^s \overline{u_\bk} v_\bk.
\]
\begin{lem}
\label{lem:Hs-bnd}
If $V\in \Hs{s}$ for some $s\in\R$, 
then $V$ satisfies Assumption~\ref{as:pot} for all $r\ge 0$ satisfying $s\ge r-1$ and $s> -2(1-r)+d/2$ resp.
\begin{equation}
    \label{eq:SobEmb}
    r \le s+1\qquad\mbox{and}\qquad r<1+\frac{s}{2}-\frac{d}{4},
\end{equation}
where $d$ is the spatial dimension. Moreover,
\[
        \Enorm{V}  \le C_{r,s} \|V\|_{\Hs{s}}.\]\end{lem}

\begin{proof}
 From \cite[Theorem 1.4.4.2]{Grisvard1985-ao},  for any $s_1,s_2 \ge t,$ such that $s_1+s_2 > t +d/2$, there exists $C_{s_1,s_2,t}>0$ such that for all $u \in \Hs{s_1}, v \in \Hs{s_2},$ then $uv \in \Hs{t}$ and 
    \[
        \| uv \|_{\Hs{t}} \le C_{s_1,s_2,t}  \| u \|_{\Hs{s_1}}  \| v \|_{\Hs{s_2}}.
    \]
    Hence if $V\in \Hs{s}$, $s > -2(1-r) + d/2,$ and $s\ge r-1$, then there exists $ C_{r,s}$ such that
    \[
        \Enorm{V} = \sup_{\psi \in \Hs{1-r}}
        \frac{\|V\psi\|_{\Hs{-1+r}}}{\|\psi\|_{\Hs{1-r}}} \le C_{r,s} \|V\|_{\Hs{s}},
    \]
which implies the estimate in the lemma. \end{proof}

\begin{proof}[Proof of Proposition~\ref{prop:UN-well-defined}]
 For each $\lax$ such that $\lax <  \ka_M$,
 $\UM(\lax)$ can be written as the product of five bounded operators: $-\PM (-\Delta + 1)^{1/2-r/2}$, $(-\Delta + 1)^{-1/2+r/2} V (-\Delta+1)^{-1/2+r/2}$, 
 $(-\Delta + 1)^{1/2-r/2} \PMp (\cH_M^\perp-\lax)^{-1} \PMp (-\Delta + 1)^{1/2-r/2}$, $(-\Delta + 1)^{-1/2+r/2} V (-\Delta+1)^{-1/2+r/2}$, $(-\Delta + 1)^{1/2-r/2} \PM$.
 Therefore, $\UM(\lax)$ is a well-defined operator.
\end{proof}

\end{document}